\documentclass{amsart}

\usepackage{savesym}
\usepackage{txfonts,stmaryrd}
\savesymbol{lrcorner}\savesymbol{ulcorner}
\usepackage{amssymb, amscd,setspace,mathtools,makecell}
\usepackage{enumerate}
\usepackage[usenames,dvipsnames]{xcolor}
\usepackage[bookmarks=true,colorlinks=true, linkcolor=MidnightBlue, citecolor=cyan]{hyperref}
\usepackage{lmodern}
\usepackage{graphicx,float}
\restoresymbol{txfonts}{lrcorner}\restoresymbol{txfonts}{ulcorner}

\usepackage{tikz}
\usetikzlibrary{arrows,calc,chains,matrix,positioning,scopes,snakes}

\def\blackbox#1#2#3#4#5{
  \pgfgetlastxy{\llx}{\lly}
  \path #1;
  \pgfgetlastxy{\w}{\h}
  \pgfmathsetlengthmacro{\urx}{\llx+\w}
  \pgfmathsetlengthmacro{\ury}{\lly+\h}
  \draw (\llx,\lly) rectangle (\urx,\ury);
  \pgfmathsetlengthmacro{\xave}{(\llx+\urx)/2}
  \pgfmathsetlengthmacro{\yave}{\ury-8}
  \node at (\xave,\yave) {#4};
  \pgfmathsetlengthmacro{\ydiff}{\ury-\lly}
  \pgfmathsetlengthmacro{\lstep}{\ydiff/(#2+1)}
  \pgfmathsetlengthmacro{\rstep}{\ydiff/(#3+1)}
  \ifnum #2=0{}\else{ 
   \foreach \l in {1,...,#2}{
    \draw [->] ($(\llx,\lly)+(-#5/2,0)+\l*(0,\lstep)$) -- ($(\llx,\lly)+(#5/2,0)+\l*(0,\lstep)$);}}\fi
  \ifnum #3=0{}\else{
   \foreach \r in {1,...,#3}{
    \draw [->] ($(\urx,\ury)+(-#5/2,0)-\r*(0,\rstep)$) -- ($(\urx,\ury)+(#5/2,0)-\r*(0,\rstep)$);}}\fi
}

\def\dashblackbox#1#2#3#4#5{
  \pgfgetlastxy{\llx}{\lly}
  \path #1;
  \pgfgetlastxy{\w}{\h}
  \pgfmathsetlengthmacro{\urx}{\llx+\w}
  \pgfmathsetlengthmacro{\ury}{\lly+\h}
  \draw [dashed] (\llx,\lly) rectangle (\urx,\ury);
  \pgfmathsetlengthmacro{\xave}{(\llx+\urx)/2}
  \pgfmathsetlengthmacro{\yave}{\ury-8}
  \node at (\xave,\yave) {#4};
  \pgfmathsetlengthmacro{\ydiff}{\ury-\lly}
  \pgfmathsetlengthmacro{\lstep}{\ydiff/(#2+1)}
  \pgfmathsetlengthmacro{\rstep}{\ydiff/(#3+1)}
  \ifnum #2=0{}\else{ 
   \foreach \l in {1,...,#2}{
    \draw [->] ($(\llx,\lly)+(-#5/2,0)+\l*(0,\lstep)$) -- ($(\llx,\lly)+(#5/2,0)+\l*(0,\lstep)$);}}\fi
  \ifnum #3=0{}\else{
   \foreach \r in {1,...,#3}{
    \draw [->] ($(\urx,\ury)+(-#5/2,0)-\r*(0,\rstep)$) -- ($(\urx,\ury)+(#5/2,0)-\r*(0,\rstep)$);}}\fi
}

\def\directarc#1#2{
  \path #1;
  \pgfgetlastxy{\lx}{\ly}
  \path #2;
  \pgfgetlastxy{\rx}{\ry}
  \pgfmathsetlengthmacro{\xave}{(\lx+\rx)/2}
  \draw #1 .. controls (\xave,\ly) and (\xave,\ry) .. #2;
}

\def\loopright#1#2#3{
  \path #1;
  \pgfgetlastxy{\ux}{\uy}
  \path #2;
  \pgfgetlastxy{\lx}{\ly}
  \pgfmathsetlengthmacro{\maxx}{max(\ux,\lx)}
  \pgfmathsetlengthmacro{\farx}{\maxx+#3}
  \draw #1 .. controls (\farx,\uy) and (\farx,\ly) .. #2;
}

\def\loopleft#1#2#3{
  \path #1;
  \pgfgetlastxy{\ux}{\uy}
  \path #2;
  \pgfgetlastxy{\lx}{\ly}
  \pgfmathsetlengthmacro{\minx}{min(\ux,\lx)}
  \pgfmathsetlengthmacro{\farx}{\minx-#3}
  \draw #1 .. controls (\farx,\uy) and (\farx,\ly) .. #2;
}

\input xy
\xyoption{all} \xyoption{poly} \xyoption{knot}\xyoption{curve}
\usepackage{xy,color}
\newcommand{\comment}[1]{}

\def\tn{\textnormal}

\def\mc{\mathcal}

\def\NN{{\mathbb N}}

\def\Hom{\tn{Hom}}

\def\Aut{\tn{Aut}}

\def\Ob{\tn{Ob}}

\def\SEL*{\tn{SEL*}}

\def\hsp{\hspace{.3in}}

\def\singleton{{\{*\}}}

\def\LoopSchema{{\parbox{.5in}{\fbox{\xymatrix{\LMO{s}\ar@(l,u)[]^f}}}}}

\def\to{\rightarrow}

\def\taking{\colon}
\def\pls{``\!+\!"}

\def\inj{\hookrightarrow}

\def\too{\longrightarrow}

\def\ss{\subseteq}

\def\iso{\cong}

\def\|{{\;|\;}}
\def\m1{{-1}}

\def\ol{\overline}
\def\ul{\underline}

\def\qeq{\mathop{=}^?}

\newcommand{\LMO}[1]{\stackrel{#1}{\bullet}}

\def\ullimit{\ar@{}[rd]|(.3)*+{\lrcorner}}
\def\urlimit{\ar@{}[ld]|(.3)*+{\llcorner}}
\def\lllimit{\ar@{}[ru]|(.3)*+{\urcorner}}
\def\lrlimit{\ar@{}[lu]|(.25)*+{\ulcorner}}
\def\ulhlimit{\ar@{}[rd]|(.3)*+{\diamond}}
\def\urhlimit{\ar@{}[ld]|(.3)*+{\diamond}}
\def\llhlimit{\ar@{}[ru]|(.3)*+{\diamond}}
\def\lrhlimit{\ar@{}[lu]|(.3)*+{\diamond}}
\newcommand{\clabel}[1]{\ar@{}[rd]|(.5)*+{#1}}
\newcommand{\TriRight}[7]{\xymatrix{#1\ar[dr]_{#2}\ar[rr]^{#3}&&#4\ar[dl]^{#5}\\&#6\ar@{}[u] |{\Longrightarrow}\ar@{}[u]|>>>>{#7}}}
\newcommand{\TriLeft}[7]{\xymatrix{#1\ar[dr]_{#2}\ar[rr]^{#3}&&#4\ar[dl]^{#5}\\&#6\ar@{}[u] |{\Longleftarrow}\ar@{}[u]|>>>>{#7}}}
\newcommand{\TriIso}[7]{\xymatrix{#1\ar[dr]_{#2}\ar[rr]^{#3}&&#4\ar[dl]^{#5}\\&#6\ar@{}[u] |{\Longleftrightarrow}\ar@{}[u]|>>>>{#7}}}

\newcommand{\arr}[1]{\ar@<.5ex>[#1]\ar@<-.5ex>[#1]}
\newcommand{\arrr}[1]{\ar@<.7ex>[#1]\ar@<0ex>[#1]\ar@<-.7ex>[#1]}
\newcommand{\arrrr}[1]{\ar@<.9ex>[#1]\ar@<.3ex>[#1]\ar@<-.3ex>[#1]\ar@<-.9ex>[#1]}
\newcommand{\arrrrr}[1]{\ar@<1ex>[#1]\ar@<.5ex>[#1]\ar[#1]\ar@<-.5ex>[#1]\ar@<-1ex>[#1]}

\newcommand{\To}[1]{\xrightarrow{#1}}
\newcommand{\Too}[1]{\xrightarrow{\ \ #1\ \ }}
\newcommand{\From}[1]{\xleftarrow{#1}}

\def\qeq{\stackrel{?}{=}}

\newcommand{\Adjoint}[4]{\xymatrix@1{{#2}\ar@<.5ex>[r]^-{#1} &{#3} \ar@<.5ex>[l]^-{#4}}}

\def\id{\tn{id}}

\def\List{\tn{List}}

\def\Sets{{\bf Sets}}

\def\Set{{\bf Set}}

\def\bhline{\Xhline{2\arrayrulewidth}}
\def\bbhline{\Xhline{2.5\arrayrulewidth}}

\def\colim{\mathop{\tn{colim}}}

\def\mcE{\mc{E}}

\def\mcO{\mc{O}}
\def\mcP{\mc{P}}

\def\mcT{\mc{T}}

\def\mcW{\mc{W}}

\def\monOb{{\blacktriangle}}

\def\Fin{{\bf Fin}}

\newcommand{\inp}[1]{{\tt in}(#1)}
\newcommand{\outp}[1]{{\tt out}(#1)}
\def\vset{{\tt vset}}
\newcommand{\vLst}[1]{\ol{#1}}
\newcommand{\Strm}[1]{{\tn{Strm}(#1)}}
\newcommand{\Hist}{\tn{Hist}}

\newcommand{\strst}[1]{\big|_{[1,#1]}} 
\newcommand{\Del}[1]{DN_{#1}}
\newcommand{\Dem}[1]{{Dm_{#1}}}
\newcommand{\Sup}[1]{{Sp_{#1}}}
\newcommand{\inDem}[1]{{in\Dem{#1}}}
\newcommand{\inSup}[1]{{in\Sup{#1}}}
\newcommand{\vinSup}[1]{\vLst{\inSup{#1}}}
\newcommand{\vinDem}[1]{\vLst{\inDem{#1}}}
\newcommand{\vSup}[1]{\vLst{\Sup{#1}}}
\newcommand{\vDem}[1]{\vLst{\Dem{#1}}}
\newcommand{\vDel}[1]{\vLst{\Del{#1}}}
\newcommand{\vinp}[1]{\vLst{\inp{#1}}}
\newcommand{\voutp}[1]{\vLst{\outp{#1}}}
\def\zipwith{\;\raisebox{3pt}{${}_\varcurlyvee$}\;}
\def\tbzipwith{\!\;\raisebox{2pt}{${}_\varcurlyvee$}\;\!}
\newcommand{\ffootnote}[2]{\hspace{#1}\footnote{#2}}

\newcommand{\disc}[1]{{\ul{#1}}}

\makeatletter\let\c@figure\c@equation\makeatother 
\newtheorem{theorem}[subsubsection]{Theorem}
\newtheorem{lemma}[subsubsection]{Lemma}
\newtheorem{proposition}[subsubsection]{Proposition}

\theoremstyle{remark}
\newtheorem{remark}[subsubsection]{Remark}
\newtheorem{example}[subsubsection]{Example}

\newtheorem{warning}[subsubsection]{Warning}
\newtheorem{question}[subsubsection]{Question}
\newtheorem{guess}[subsubsection]{Guess}

\theoremstyle{definition}
\newtheorem{definition}[subsubsection]{Definition}

\newtheorem{ann}[subsubsection]{Announcement}
\newenvironment{announcement}{\begin{ann}}{\hspace*{\fill}$\lozenge$\end{ann}}

\usepackage[paperwidth=8.5in,paperheight=11in,text={5.7in,8in},centering]{geometry}
\usepackage{lscape}

\begin{document}

\title[The operad of temporal wiring diagrams]{The operad of temporal wiring diagrams: formalizing a graphical language for discrete-time processes}

\author{Dylan Rupel}

\address{Department of Mathematics, Northeastern University, Boston, MA 02115}

\email{d.rupel@neu.edu}

\author{David I. Spivak}

\address{Department of Mathematics, Massachusetts Institute of Technology, Cambridge MA 02139}

\email{dspivak@math.mit.edu}

\thanks{Spivak acknowledges support by ONR grant N000141310260.}

\begin{abstract}

We investigate the hierarchical structure of processes using the mathematical theory of operads. Information or material enters a given process as a stream of inputs, and the process converts it to a stream of outputs. Output streams can then be supplied to other processes in an organized manner, and the resulting system of interconnected processes can itself be considered a macro process. To model the inherent structure in this kind of system, we define an operad $\mcW$ of black boxes and directed wiring diagrams, and we define a $\mcW$-algebra $\mcP$ of processes (which we call propagators, after \cite{RS}). Previous operadic models of wiring diagrams (e.g. \cite{Sp2}) use undirected wires without length, useful for modeling static systems of constraints, whereas we use directed wires with length, useful for modeling dynamic flows of information. We give multiple examples throughout to ground the ideas.

\end{abstract}

\maketitle

\tableofcontents

\section{Introduction}

Managing processes is inherently a hierarchical and self-similar affair. Consider the case of preparing a batch of cookies, or if one prefers, the structurally similar case of manufacturing a pharmaceutical drug. To make cookies, one generally follows a recipe, which specifies a process that is undertaken by subdividing it as a sequence of major steps. These steps can be performed in series or in parallel. The notion of self-similarity arises when we realize that each of these major steps can itself be viewed as a process, and thus it can also be subdivided into smaller steps. For example, procuring the materials necessary to make  cookies involves getting oneself to the appropriate store, selecting the necessary materials, paying for them, etc., and each of these steps is itself a simpler process. 

Perhaps every such hierarchy of nesting processes must touch ground at the level of atomic detail. Hoping that the description of processes within processes would not continue {\em ad infinitum} may have led humanity to investigate matter and motion at the smallest level possible. This investigation into atomic and quantum physics has yielded tremendous technological advances, such as the invention of the microchip. 

Working on the smallest possible scale is not always effective, however. It appears that the planning and execution of processes benefits immensely from hierarchical chunking. To write a recipe for cookies at the level of atomic detail would be expensive and useless. Still, when executing our recipe, the decision to add salt will initiate an unconscious procedure, by which signals are sent from the brain to the muscles of the arm, on to individual cells, and so on until actual atoms move through space and ``salt has been added". Every player in the larger cookie-making endeavor understands the current demand (e.g. to add salt) as a procedure that makes sense at his own level of granularity. The decision to add salt is seen as a mundane (low-level) job in the context of planning to please ones girlfriend by baking cookies; however this same decision is seen as an abstract (high-level) concept in the context of its underlying performance as atomic movements.

For designing complex processes, such as those found in manufacturing automobiles or in large-scale computer programming, the architect and engineers must be able to change levels of abstraction with ease. In fact, different engineers working on the same project are often thinking about the same basic structures, but in different terms. They are most effective when they can chunk the basic structures as they see fit. 

A person who studies a supply chain in terms of the function played by each {\em chain member} should be able to converse coherently with a person who studies the same supply chain in terms of the contracts and negotiations that exist at each {\em chain link}. These are two radically different viewpoints on the same system, and it is useful to be able to switch fluidly between them. Similarly, an engineer designing a system's hardware must be able to converse with an engineer working on the system's software. Otherwise, small perturbations made by one of them will be unexpected by the other, and this can lead to major problems.

The same types of issues emerge whether one is concerned with manufacturers in a supply chain, neurons in a functional brain region, modules in a computer program, or steps in a recipe. In each case, what we call {\em propagators} (after \cite{RS}) are being arranged into a system that is itself a propagator at a higher level. The goal of this paper is to provide a mathematical basis for thinking about this kind of problem. We offer a formalism that describes the hierarchical and self-similar nature of a certain kind of wiring diagram. 

A similar kind of wiring diagram was described in \cite{Sp2}, the main difference being that the present one is built for time-based processes whereas the one in \cite{Sp2} was built for static relations. In the present work we take the notion of time (or one may say distance) seriously. We go through considerable effort to integrate a notion of time and distance into the fundamental architecture of our description, by emphasizing that communication channels have a length, i.e. that communication takes time.

Design choices such as these greatly affect the behavior of our model, and ours was certainly not the only viable choice. We hope that the basic idea we propose will be a basis upon which future engineers and mathematicians will improve. For the time being, we may at least say that the set of rules we propose for our wiring diagrams roughly conform with the IDEF0 standard set by the National Institute of Standards and Technology \cite{NIST}. The main differences are that in our formalism, 
\begin{itemize}
\item wires can split but not merge (each merging must occur within a particular box), 
\item feedback loops are allowed, 
\item the so-called control and mechanism arrows are subsumed into input and output arrows, and 
\item the rules for and meaning of hierarchical composition is made explicit. 
\end{itemize}

The basic picture to have in mind for our wiring diagrams is the following:
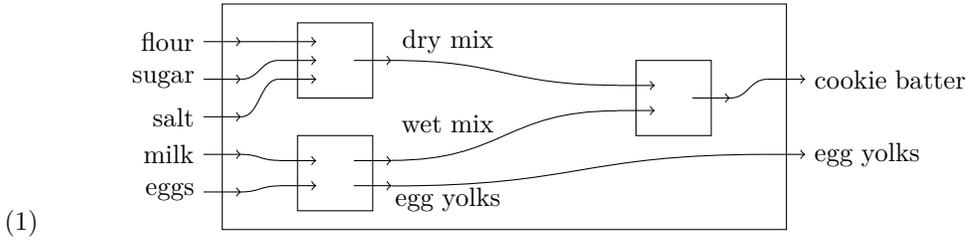
\begin{align}\label{dia:cookies}
 \begin{tikzpicture}
  \path (0,0);
  \blackbox{(7.5,3)}{5}{2}{}{0.5}
  \draw (-0.25,0.5) node[left] {eggs};
  \draw (-0.25,1) node[left] {milk};
  \draw (-0.25,1.5) node[left] {salt};
  \draw (-0.25,2) node[left] {sugar};
  \draw (-0.25,2.5) node[left] {flour};
  \draw (3,2.5) node{dry mix};
  \draw (3,1.4) node{wet mix};
  \draw (3,.4) node{egg yolks};
  \draw (7.75,1) node[right] {egg yolks};
  \draw (7.75,2) node[right] {cookie batter};
  \path (1,1.75);
  \blackbox{(1,1)}{3}{1}{}{0.5}
  \path (1,0.25);
  \blackbox{(1,1)}{2}{2}{}{0.5}
  \path (5.5,1.25);
  \blackbox{(1,1)}{2}{1}{}{0.5}
  \directarc{(0.25,0.5)}{(0.75,0.58)}
  \directarc{(0.25,1)}{(0.75,0.92)}
  \directarc{(0.25,1.5)}{(0.75,2)}
  \directarc{(0.25,2)}{(0.75,2.25)}
  \directarc{(0.25,2.5)}{(0.75,2.5)}
  \directarc{(2.25,0.58)}{(7.25,1)}
  \directarc{(2.25,0.92)}{(5.25,1.58)}
  \directarc{(2.25,2.25)}{(5.25,1.92)}
  \directarc{(6.75,1.75)}{(7.25,2)}
 \end{tikzpicture}
\end{align}
In this picture we see an exterior box, some interior boxes, and a collection of directed wires. These directed wires transport some type of product from the export region of some box to the import region of some box. In (\ref{dia:cookies}) we have a supply chain involving three propagators, one of whom imports flour, sugar, and salt and exports dry mixture, and another of whom imports eggs and milk and exports egg yolks and wet mixture. The dry mixture and the wet mixture are then transported to a third propagator who exports cookie batter. The whole system itself constitutes a propagator that takes five ingredients and produces cookie batter and egg yolks. 

The formalism we offer in this paper is based on a mathematical structure called an {\em operad} (more precisely, a {\em symmetric colored operad}), chosen because they capture the self-similar nature of wiring diagrams. The rough idea is that if we have a wiring diagram and we insert wiring diagrams into each of its interior boxes, the result is a new wiring diagram.
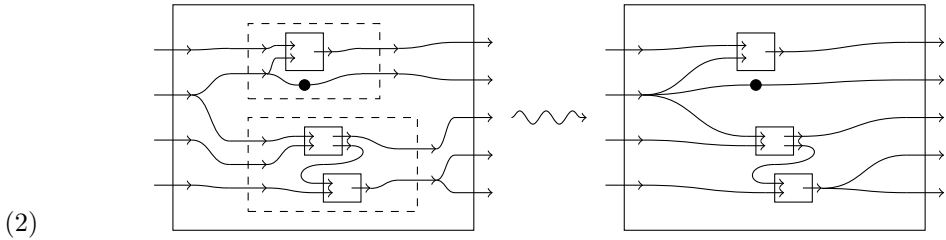
\begin{align}
 \begin{tikzpicture}
  \path (0,0);
  \blackbox{(4,3)}{4}{5}{}{0.5}
  \path (1,1.75);
  \dashblackbox{(1.75,1)}{2}{2}{}{0.5}
  \directarc{(0.25,1.8)}{(0.75,2.08)}
  \directarc{(0.25,2.4)}{(0.75,2.42)}
  \directarc{(3,2.08)}{(3.75,2)}
  \directarc{(3,2.42)}{(3.75,2.5)}
  \path (1.5,2.125);
  \blackbox{(0.5,0.5)}{2}{1}{}{0.25}
  \directarc{(1.25,2.08)}{(1.375,2.2916)}
  \directarc{(1.25,2.42)}{(1.375,2.4584)}
  \directarc{(2.125,2.375)}{(2.5,2.42)}
  \filldraw[black] (1.75,1.9325) circle (2pt);
  \directarc{(1.25,2.08)}{(1.65,1.9325)}
  \directarc{(1.65,1.9325)}{(1.7,1.9325)}
  \directarc{(1.75,1.9325)}{(2.5,2.08)}
  \path (1,0.25);
  \dashblackbox{(2.25,1.25)}{3}{2}{}{0.5}
  \directarc{(0.25,0.6)}{(0.75,0.5625)}
  \directarc{(0.25,1.2)}{(0.75,0.875)}
  \directarc{(0.25,1.8)}{(0.75,1.1875)}
  \directarc{(3.5,0.6666)}{(3.75,0.5)}
  \directarc{(3.5,0.6666)}{(3.75,1)}
  \directarc{(3.5,1.0833)}{(3.75,1.5)}
  \path (2,0.375);
  \blackbox{(0.5,0.375)}{2}{1}{}{0.25}
  \directarc{(1.25,0.5625)}{(1.875,0.5)}
  \directarc{(2.625,0.5575)}{(3,0.6666)}
  \loopright{(2.375,1.125)}{(2,0.875)}{7}
  \loopleft{(2,0.875)}{(1.875,0.625)}{7}
  \path (1.75,1);
  \blackbox{(0.5,0.375)}{2}{2}{}{0.25}
  \directarc{(1.25,0.875)}{(1.625,1.125)}
  \directarc{(1.25,1.1875)}{(1.625,1.25)}
  \directarc{(2.375,1.25)}{(3,1.0833)}
  \draw [->,snake=snake] (4.5,1.5) -- (5.5,1.5);
  \path (6,0);
  \blackbox{(4,3)}{4}{5}{}{0.5}
  \path (7.5,2.125);
  \blackbox{(0.5,0.5)}{2}{1}{}{0.25}
  \directarc{(6.25,1.8)}{(7.375,2.2916)}
  \directarc{(6.25,2.4)}{(7.375,2.4584)}
  \directarc{(8.125,2.375)}{(9.75,2.5)}
  \filldraw[black] (7.75,1.9325) circle (2pt);
  \directarc{(6.25,1.8)}{(7.65,1.9325)}
  \directarc{(7.65,1.9325)}{(7.7,1.9325)}
  \directarc{(7.75,1.9325)}{(9.75,2)}
  \path (8,0.375);
  \blackbox{(0.5,0.375)}{2}{1}{}{0.25}
  \directarc{(6.25,0.6)}{(7.875,0.5)}
  \directarc{(8.625,0.5575)}{(9.75,0.5)}
  \directarc{(8.625,0.5575)}{(9.75,1)}
  \loopright{(8.375,1.125)}{(8,0.875)}{7}
  \loopleft{(8,0.875)}{(7.875,0.625)}{7}
  \path (7.75,1);
  \blackbox{(0.5,0.375)}{2}{2}{}{0.25}
  \directarc{(6.25,1.2)}{(7.625,1.125)}
  \directarc{(6.25,1.8)}{(7.625,1.25)}
  \directarc{(8.375,1.25)}{(9.75,1.5)}
 \end{tikzpicture}
\end{align}
We will make explicit what constitutes a box, what constitutes a wiring diagram (WD), and how inserting WDs into a WD constitutes a new WD. Like Russian dolls, we may have a nesting of WDs inside of WDs inside of WDs, etc. We will prove an associativity law that guarantees that no matter how deeply our Russian dolls are nested, the resulting WD is well-defined. Once all this is done, we will have an operad $\mcW$.

To make this directed wiring diagrams operad $\mcW$ useful, we will take our formalism to the next logical step and provide an {\em algebra} on $\mcW$. This algebra $\mcP$ encodes our application to process management by telling us what fits in the boxes and how to use wiring diagrams to build more complex systems out of simpler components. More precisely, the algebra $\mcP$ makes explicit 
\begin{itemize}
\item the set of things that can go in every box, namely the set of propagators, and
\item a method for taking a wiring diagram and a propagator for each of its interior boxes and producing a propagator for the exterior box. 
\end{itemize}
To prove that we have an algebra, we will show that no matter how one decides to group the various internal propagators, the behavior of the resulting system is unchanged.

Operads were invented in the 1970s by \cite{May} and \cite{BV} in order to encode the relationship between various operations they noticed taking place in the mathematical field of  algebraic topology. At the moment we are unconcerned with topological properties of our operads, but the formalism grounds the picture we are trying to get across. For more on operads, see \cite{Lei}.

\subsection{Structure of the paper}

In Section \ref{sec:operad} we discuss operads. In Section \ref{sec:operad def} we give the mathematical definition of operads and some examples. In Section \ref{sec:operad announcements} we propose the operad of interest, namely $\mcW$, the operad of directed wiring diagrams. We offer an example wiring diagram in Section \ref{sec:ground W} that will run throughout the paper and eventually output the Fibonacci sequence. In Section \ref{sec:operad requirements} we prove that $\mcW$ has the required properties so that it is indeed an operad. 

In Section \ref{sec:algebra} we discuss algebras on an operad. In Section \ref{sec:algebra def} we give the mathematical definition of algebras. In Sections \ref{sec:lists streams props} we discuss some preliminaries on lists and define our notion of historical propagators, which we will then use in \ref{sec:algebra announcements} where we propose the $\mcW$-algebra  of interest, the algebra of propagators. In Section \ref{sec:algebra requirements} we prove that $\mcP$ has the required properties so that it is indeed a $\mcW$-algebra. 

We expect the majority of readers to be most interested in the running examples sections, Sections \ref{sec:ground W} and \ref{sec:ground P}. Readers who want more details, e.g. those who may wish to write code for propagators, will need to read Sections \ref{sec:operad announcements}, \ref{sec:algebra announcements}. The proof that our algebra satisfies the necessary requirements is technical; we expect only the most dedicated readers to get through it. Finally, in Section \ref{sec:future work} we discuss some possibilities for future work in this area.

The remainder of the present section is devoted to our notational conventions (Section \ref{sec:notation}) and our acknowledgments (\ref{sec:acknowledgements}).

\subsection{Notation and background}\label{sec:notation}

Here we describe our notational conventions. These are only necessary for readers who want a deep understanding of the underlying mathematics. Such readers are assumed to know some basic category theory. For mathematicians we recommend \cite{Awo} or \cite{Mac}, for computer scientists we recommend \cite{Awo} or \cite{BW}, and for a general audience we recommend \cite{Sp1}. 

We will primarily be concerned only with the category of small sets, which we denote by $\Set$, and some related categories. We denote by $\Fin\ss\Set$ the full subcategory spanned by finite sets. We often use the symbol $n\in\Ob(\Fin)$ to denote a finite set, and may speak of elements $i\in n$. The cardinality of a finite set is a natural number, denoted $|n|\in\NN$. In particular, we consider $0$ to be a natural number.

Suppose given a finite set $n$ and a function $X\taking n\to\Ob(\Set)$, and let $\amalg_{i\in I}X(i)$ be the disjoint union. Then there is a canonical function $\pi_X\taking\amalg_{i\in n}X(i)\too n$ which we call the {\em component projection}. We use almost the same symbol in a different context; namely, for any function $s\taking m\to n$ we denote the {\em $s$-coordinates projection} by 
$$\pi_s\taking\prod_{i\in n}X(i)\too\prod_{j\in m}X(s(j)).$$
In particular, if $i\in n$ is an element, we consider it as a function $i\taking\{*\}\to n$ and write $\pi_i\taking\prod_{i\in n}X(i)\to X(i)$ for the usual $i$th coordinate projection.

A {\em pointed set} is a pair $(S,s)$ where $S\in\Ob(\Set)$ is a set and $s\in S$ is a chosen element, called the {\em base point}. In particular a pointed set cannot be empty. Given another pointed set $(T,t)$, a {\em pointed function from $(S,s)$ to $(T,t)$} consists of a function $f\taking S\to T$ such that $f(s)=t$. We denote the category of pointed sets by $\Set_*$. There is a forgetful functor $\Set_*\to\Set$ which forgets the basepoint; it has a left adjoint which adjoins a free basepoint $X\mapsto X\amalg\{*\}$. We often find it convenient not to mention basepoints; if we speak of a set $X$ as though it is pointed, we are actually speaking of $X\amalg\{*\}$. If $S,S'$ are pointed sets then the product $S\times S'$ is also naturally pointed, with basepoint $(*,*)$, again denoted simply by $*$. 

We often speak of functions $n\to\Ob(\Set_*)$, where $n$ is a finite set. Of course, $\Ob(\Set_*)$ is not itself a small set, but using the theory of Grothendieck universes \cite{Bou}, this is not a problem. It will be even less of a problem in applications.

\subsection{Acknowledgements}\label{sec:acknowledgements}

David Spivak would like to thank Sam Cho as well as the NIST community, especially Al Jones and Eswaran Subrahmanian. Special thanks go to Nat Stapleton for many valuable conversations in which substantial progress was made toward subjects quite similar to the ones we discuss here.

Dylan Rupel would like to thank Jason Isbell and Kiyoshi Igusa for many useful discussions.

\section{$\mcW$, the operad of directed wiring diagrams}\label{sec:operad}

In this section we will define the operad $\mcW$ of black boxes and directed wiring diagrams (WDs). It governs the forms that a black box can take, the rules that a WD must follow, and the formula for how the substitution of WDs into a WD yields a WD. There is no bound on the depth to which wiring diagrams can be nested. That is, we prove an associative law which roughly says that the substitution formula is well-defined for any degree of nesting, shallow or deep. 

We will use the operad $\mcW$ to discuss the hierarchical nature of processes. Each box in our operad will be filled with a process, and each wiring diagram will effectively build a complex process out of simpler ones. However, this is not strictly a matter of the operad $\mcW$ but of an {\em algebra on $\mcW$}. This algebra will be discussed in Section \ref{sec:algebra}.

The present section is organized as follows. First, in Section \ref{sec:operad def} we give the technical definition of the term {\em operad} and a few examples. In Section \ref{sec:operad announcements} we propose our operad $\mcW$ of wiring diagrams. It will include drawings that should clarify the matter. In Section \ref{sec:ground W} we present an example that will run throughout the paper and end up producing the Fibonacci sequence. This section is recommended especially to the more category-theoretically shy reader. Finally, in Section \ref{sec:operad requirements} we give a technical proof that our proposal for $\mcW$ satisfies the requirements for being a true operad, i.e. we establish the well-definedness of repeated substitution as discussed above.

\subsection{Definition and basic examples of operads}\label{sec:operad def}

Before we begin, we should give a warning about our use of the term ``operad".

\begin{warning}\label{warn:operad}

Throughout this paper, we use the word {\em operad} to mean what is generally called a {\em symmetric colored operad} or a {\em symmetric multicategory}. This abbreviated nomenclature is not new, for example it is used in \cite{Lur}. Hopefully no confusion will arise. For a full treatment of operads, multicategories, and how they fit into a larger mathematical context, see \cite{Lei}.

\end{warning}

Most of Section \ref{sec:operad def} is recycled material, taken almost verbatim from \cite{Sp2}. We repeat it here for the convenience of the reader.

\begin{definition}\label{def:operad}

An {\em operad} $\mcO$ is defined as follows: One announces some constituents (A. objects, B. morphisms, C. identities, D. compositions) and proves that they satisfy some requirements (1. identity law, 2. associativity law). Specifically, 
\begin{enumerate}[\hsp A.]
\item one announces a collection $\Ob(\mcO)$, each element of which is called an {\em object} of $\mcO$.
\item for each object $y\in\Ob(\mcO)$, finite set $n\in\Ob(\Fin)$, and $n$-indexed set of objects $x\taking n\to\Ob(\mcO)$, one announces a set $\mcO_n(x;y)\in\Ob(\Set)$. Its elements are called {\em morphisms from $x$ to $y$} in $\mcO$. 
\item for every object $x\in\Ob(\mcO)$, one announces a specified morphism denoted $\id_x\in\mcO_1(x;x)$ called {\em the identity morphism on $x$}.
\item Let $s\taking m\to n$ be a morphism in $\Fin$. Let $z\in\Ob(\mcO)$ be an object, let $y\taking n\to\Ob(\mcO)$ be an $n$-indexed set of objects, and let $x\taking m\to\Ob(\mcO)$ be an $m$-indexed set of objects. For each element $i\in n$, write $m_i:=s^\m1(i)$ for the pre-image of $s$ under $i$, and write $x_i=x|_{m_i}\taking m_i\to\Ob(\mcO)$ for the restriction of $x$ to $m_i$. Then one announces a function 
\begin{align}\label{dia:composition formula}
\circ\taking\mcO_n(y;z)\times\prod_{i\in n}\mcO_{m_i}(x_i;y(i))\too\mcO_{m}(x;z),
\end{align} 
called {\em the composition formula} for $\mcO$.
\end{enumerate}
Given an $n$-indexed set of objects $x\taking n\to\Ob(\mcO)$ and an object $y\in\Ob(\mcO)$, we sometimes abuse notation and denote the set of morphisms from $x$ to $y$ by $\mcO(x_1,\ldots,x_n;y)$.
\ffootnote{-4pt}{There are three abuses of notation when writing $\mcO(x_1,\ldots,x_n;y)$, which we will fix one by one. First, it confuses the set $n\in\Ob(\Fin)$ with its cardinality $|n|\in\NN$. But rather than writing $\mcO(x_1,\ldots,x_{|n|};y)$, it would be more consistent to write $\mcO(x(1),\ldots,x(|n|);y)$, because we have assigned subscripts another meaning in D. However, even this notation unfoundedly suggests that the set $n$ has been endowed with a linear ordering, which it has not. This may be seen as a more serious abuse, but see Remark \ref{rem:symmetry}.}
We may write $\Hom_\mcO(x_1,\ldots,x_n;y)$, in place of $\mcO(x_1,\ldots,x_n;y)$, when convenient. We can denote a morphism $\phi\in\mcO_n(x;y)$ by $\phi\taking x\to y$ or by $\phi\taking (x_1,\ldots,x_n)\to y$; we say that each $x_i$ is a {\em domain object} of $\phi$ and that $y$ is the {\em codomain object} of $\phi$. We use infix notation for the composition formula, e.g. writing $\psi\circ(\phi_1,\ldots,\phi_n)$.

These constituents (A,B,C,D) must satisfy the following requirements:
\begin{enumerate}[\hsp 1.]
\item for every $x_1,\ldots,x_n,y\in\Ob(\mcO)$ and every morphism $\phi\taking(x_1,\ldots,x_n)\to y$, we have
$$\phi\circ(\id_{x_1},\ldots,\id_{x_n})=\phi\hsp\tn{and}\hsp\id_y\circ\phi=\phi;$$
\item Let $m\To{s}n\To{t}p$ be composable morphisms in $\Fin$. Let $z\in\Ob(\mcO)$ be an object, let $y\taking p\to\Ob(\mcO)$, $x\taking n\to\Ob(\mcO)$, and $w\taking m\to\Ob(\mcO)$ respectively be a $p$-indexed, $n$-indexed, and $m$-indexed set of objects. For each $i\in p$, write $n_i=t^\m1(i)$ for the pre-image and $x_i\taking n_i\to\Ob(\mcO)$ for the restriction. Similarly, for each $k\in n$ write $m_k=s^\m1(k)$ and $w_k\taking m_k\to\Ob(\mcO)$; for each $i\in p$, write $m_{i,-}=(t\circ s)^\m1(i)$ and $w_{i,-}\taking m_{i,-}\to\Ob(\mcO)$; for each $j\in n_i$, write $m_{i,j}:=s^\m1(j)$ and $w_{i,j}\taking m_{i,j}\to\Ob(\mcO)$. Then the diagram below commutes:
$$\hspace{-1in}\xymatrix@=18pt{
&
{\hspace{.9in}\color{white}\prod\color{black}}
\save[]+<0cm,0cm>*\txt<30pc>{$
\mcO_p(y;z)\times\prod_{i\in p}\mcO_{n_i}(x_i;y(i))\times\prod_{i\in p,\ j\in n_i}\mcO_{m_{i,j}}(w_{i,j};x_i(j))
$}
\ar[rd]\ar[ld]
\restore\\
{\hspace{1in}\color{white}\prod\color{black}}
\save[]+<.4cm,0cm>*\txt<30pc>{$
\mcO_n(x;z)\times\prod_{k\in n}\mcO_{m_k}(w_k;x(k))
$}
\ar[dr]
\restore&&
{\hspace{1in}\color{white}\prod\color{black}}
\save[]+<-.3cm,0cm>*\txt<30pc>{$
\mcO_p(y;z)\times\prod_{i\in p}\mcO_{m_{i,-}}(w_{i,-};y(i))
$}
\ar[dl]
\restore\\
&\mcO_m(w;z)
}
$$

\end{enumerate}

\end{definition}

\begin{remark}\label{rem:symmetry}

In this remark we will discuss the abuse of notation in Definition \ref{def:operad} and how it relates to an action of a symmetric group on each morphism set in our definition of operad. We follow the notation of Definition \ref{def:operad}, especially following the use of subscripts in the composition formula.

Suppose that $\mcO$ is an operad, $z\in\Ob(\mcO)$ is an object, $y\taking n\to\Ob(\mcO)$ is an $n$-indexed set of objects, and $\phi\taking y\to z$ is a morphism. If we linearly order $n$, enabling us to write $\phi\taking (y(1),\ldots,y(|n|))\to z$, then changing the linear ordering amounts to finding an isomorphism of finite sets $\sigma\taking m\To{\iso} n$, where $|m|=|n|$. Let $x=y\circ\sigma$ and for each $i\in n$, note that $m_i=\sigma^\m1(\{i\})=\{\sigma^\m1(i)\}$, so $x_i=x|_{\sigma^\m1(i)}=y(i)$. Taking $\id_{x_i}\in\mcO_{m_i}(x_i;y(i))$ for each $i\in n$, and using the identity law, we find that the composition formula induces a bijection $\mcO_n(y;z)\To{\iso}\mcO_m(x;z)$, which we might denote by 
$$\sigma\taking\mcO(y(1),y(2),\ldots,y(n);z)\iso\mcO\big(y(\sigma(1)),y(\sigma(2)),\ldots,y(\sigma(n));z\big).$$
In other words, there is an induced group action of $\Aut(n)$ on $\mcO_n(y(1),\ldots,y(n);z)$, where $\Aut(n)$ is the group of permutations of an $n$-element set.

Throughout this paper, we will permit ourselves to abuse notation and speak of morphisms $\phi\taking (x_1,x_2,\ldots,x_n)\to y$ for a natural number $n\in\NN$, without mentioning the abuse inherent in choosing an order, so long as it is clear that permuting the order of indices would not change anything up to canonical isomorphism.

\end{remark}

\begin{example}\label{ex:Sets}

We define the operad of sets, denoted $\Sets$, as follows. We put $\Ob(\Sets):=\Ob(\Set)$. Given a natural number $n\in\NN$ and objects $X_1,\ldots,X_n, Y\in\Ob(\Sets)$, we define 
$$\Sets(X_1,X_2,\ldots,X_n;Y):=\Hom_\Set(X_1\times X_2\times\cdots\times X_n,Y).$$ 
For any $X\in\Ob(\Sets)$ the identity morphism $\id_X\taking X\to X$ is the same identity as that in $\Set$. 

The composition formula is as follows. Suppose given a set $Z\in\Ob(\Set)$, a finite set $n\in\Ob(\Fin)$, for each $i\in n$ a set $Y_i\in\Ob(\Set)$ and a finite set $m_i\in\Ob(\Fin)$, and for each $j\in m_i$ a set $X_{i,j}\in\Ob(\Set)$. Suppose furthermore that we have composable morphisms: a function $g\taking\prod_{i\in n}Y_i\to Z$ and for each $i\in n$ a function $f_i\taking \prod_{j\in m_i}X_{i,j}\to Y_i$. Let $m=\amalg_im_i$. We need a function $\prod_{j\in m} X_j\to Z$, which we take to be the composite 
$$\prod_{i\in n}\prod_{j\in m_i}X_{i,j}\Too{\prod_{i\in n}f_i}\prod_{i\in n}Y_i\Too{g}Z.$$

It is not hard to see that this composition formula is associative.

\end{example}

\begin{example}\label{ex:commutative operad}

The {\em commutative operad} $\mcE$ has one object, say $\Ob(\mcE)=\{\monOb\}$, and for each $n\in\NN$ it has a single $n$-ary morphism, $\mcE_n(\blacktriangle,\ldots,\blacktriangle;\blacktriangle)=\{\mu_n\}$. 

\end{example}

\subsection{The announced structure of the wiring diagrams operad $\mcW$}\label{sec:operad announcements}

To define our operad $\mcW$, we need to announce its structure, i.e. 
\begin{itemize}
\item define what constitutes an object of $\mcW$,
\item define what constitutes a morphism of $\mcW$, 
\item define the identity morphisms in $\mcW$, and
\item the formula for composing morphisms of $\mcW$.
\end{itemize}

For each of these we will first draw and describe a picture to have in mind, then give a mathematical definition. In Section \ref{sec:operad requirements} we will prove that the announced structure has the required properties.

\subsubsection{Objects are black boxes}

Each object $X$ will be drawn as a box with input arrows entering on the left of the box and output arrows leaving from the right of the box. The arrows will be called {\em wires}. All input and output wires will be drawn across the corresponding vertical wall of the box.
   \begin{align}\label{dia:object in W}
    \begin{tikzpicture}
     \draw (0,0) rectangle (1,1);
     \draw [->] (-0.25,0.25) -- (0.25,0.25);
     \draw [->] (-0.25,0.5) -- (0.25,0.5);
     \draw [->] (-0.25,0.75) -- (0.25,0.75);
     \draw (-0.8,0.5) node{$\inp{X}$};
     \draw [->] (0.75,0.33) -- (1.25,0.33);
     \draw [->] (0.75,0.67) -- (1.25,0.67);
     \draw (2,0.5) node{$\outp{X}$};
    \end{tikzpicture}
   \end{align}
Each wire is also assigned a set of values that it can carry, and this set can be written next to the wire, or the wires may be color coded. See Example \ref{ex:black box} below. As above, we often leave off the values assignment in pictures for readability reasons. 

\begin{announcement}[Objects of $\mcW$]
An object $X\in\Ob(\mcW)$ is called a {\em black box}, or {\em box} for short. It consists of a tuple $X:=(\inp{X},\outp{X},\vset)$, where 
\begin{itemize}
 \item $\inp{X}\in\Ob(\Fin)$ is a finite set, called the set of {\em input wires to $X$},
 \item $\outp{X}\in\Ob(\Fin)$ is a finite set, called the set of {\em output wires from $X$}, and
 \item $\vset(X)\taking\inp{X}\amalg\outp{X}\to\Ob(\Set_*)$ is a function, called the {\em values assignment for $X$}. For each wire $i\in\inp{X}\amalg\outp{X}$, we call $\vset(i)\in\Ob(\Set_*)$ the set of {\em values assigned to wire $i$}, and we call its basepoint element the {\em default value} on wire $i$.
\end{itemize}
\end{announcement}

\begin{example}\label{ex:black box}

We may take $X=(\{1\},\{2,3\},\vset)$, where $\vset\taking\{1,2,3\}\to\Ob(\Set_*)$ is given by $\vset(1)=\NN$, $\vset(2)=\NN$, and $\vset(3)=\{a,b,c\}$.
\ffootnote{-2pt}{
The functor $\vset$ is supposed to assign pointed sets to each wire, but no base points are specified in the description above. As discussed in Section \ref{sec:notation}, in this case we really have $\vset(1)=\NN\amalg\{*\}$, $\vset(2)=\NN\amalg\{*\}$, and $\vset(3)=\{a,b,c\}\amalg\{*\}$, where $*$ is the default value.
}
We would draw $X$ as follows.
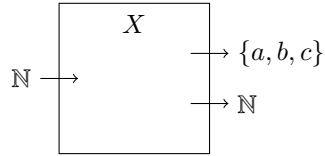
\begin{center}
\begin{tikzpicture}
\path (0,0);
  \blackbox{(2,2)}{1}{2}{$X$}{0.5}
  \draw (-0.25,1) node[left] {$\NN$};
  \draw (2.25,1.33) node[right] {$\{a,b,c\}$};
  \draw (2.25,.66) node[right] {$\NN$};  
\end{tikzpicture}
\end{center}
  
The input wire carries natural numbers, as does one of the output wires, and the other output wire carries letters $a,b,c$.

\end{example}

\subsubsection{Morphisms are directed wiring diagrams}

Given black boxes $Y_1,\ldots,Y_n\in\Ob(\mcW)$ and a black box $Z\in\Ob(\mcW)$, we must define the set $\mcW_n(Y;Z)$ of wiring diagrams (WDs) of type $Y_1,\ldots,Y_n\to Z$. Such a wiring diagram can be taken to denote a way to wire black boxes $Y_1,\ldots,Y_n$ together to form a larger black box $Z$. A typical such wiring diagram is shown below:
 
   
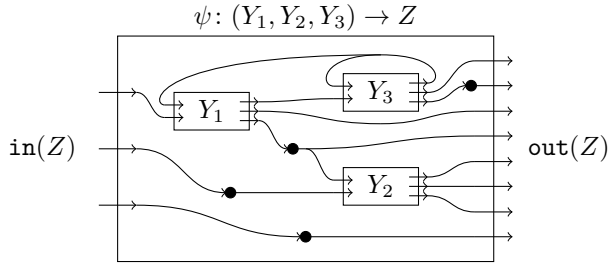
\begin{align}\label{dia:morphism in W}
    \begin{tikzpicture}
     \draw(2.5,3.25) node{$\psi\taking(Y_1,Y_2,Y_3)\to Z$};
     \draw (0,0) rectangle (5,3);
     \draw (-1,1.5) node{$\inp{Z}$};
     \draw [->] (-0.25,0.75) -- (0.25,0.75);
     \draw [->] (-0.25,1.5) -- (0.25,1.5);
     \draw [->] (-0.25,2.25) -- (0.25,2.25);
     \draw (6,1.5) node{$\outp{Z}$};
     \draw [->] (4.75,0.33) -- (5.25,0.33);
     \draw [->] (4.75,0.66) -- (5.25,0.66);
     \draw [->] (4.75,1) -- (5.25,1);
     \draw [->] (4.75,1.33) -- (5.25,1.33);
     \draw [->] (4.75,1.67) -- (5.25,1.67);
     \draw [->] (4.75,2) -- (5.25,2);
     \draw [->] (4.75,2.34) -- (5.25,2.34);
     \draw [->] (4.75,2.67) -- (5.25,2.67);
     \draw (1.25,2) node{$Y_1$};
     \draw (0.75,1.75) rectangle (1.75,2.25);
     \draw (3.5,1) node{$Y_2$};
     \draw (3,0.75) rectangle (4,1.25);
     \draw (3.5,2.25) node{$Y_3$};
     \draw (3,2) rectangle (4,2.5);
     \draw [->] (0.625,1.9166) -- (0.875,1.9166);
     \draw [->] (0.625,2.0834) -- (0.875,2.0834);
     \draw [->] (2.875,0.9166) -- (3.125,0.9166);
     \draw [->] (2.875,1.0834) -- (3.125,1.0834);
     \draw [->] (2.875,2.1666) -- (3.125,2.1666);
     \draw [->] (2.875,2.3334) -- (3.125,2.3334);
     \draw [->] (1.625,1.875) -- (1.875,1.875);
     \draw [->] (1.625,2) -- (1.875,2);
     \draw [->] (1.625,2.125) -- (1.875,2.125);
     \draw [->] (3.875,0.875) -- (4.125,0.875);
     \draw [->] (3.875,1) -- (4.125,1);
     \draw [->] (3.875,1.125) -- (4.125,1.125);
     \draw [->] (3.875,2.125) -- (4.125,2.125);
     \draw [->] (3.875,2.25) -- (4.125,2.25);
     \draw [->] (3.875,2.375) -- (4.125,2.375);
     \draw [->](4.125,2.125) .. controls (4.43875,2.125) and (4.43875,2.34) .. (4.65,2.34);%
     \filldraw[black] (1.5,0.9166) circle (2pt)
                    (2.33,1.5) circle (2pt)
                    (4.7,2.34) circle (2pt)
                    (2.5,0.33) circle (2pt);
     \draw [->] (0.25,0.75) .. controls (1.375,0.75) and (1.375,0.33) .. (2.445,0.33);
     \draw (2.5,0.33) -- (4.75,0.33);
     \draw [->] (0.25,1.5) .. controls (0.875,1.5) and (0.875,0.9166) .. (1.445,0.9166);
     \draw (1.5,0.9166) -- (2.875,0.9166);
     \draw (0.25,2.25) .. controls (0.4375,2.25) and (0.4375,1.9166) .. (0.625,1.9166);
     \draw (1.875,1.875) .. controls (2.0625,1.875) and (2.0625,1.5).. (2.25,1.5);
     \draw [->] (2.25,1.5) -- (2.275,1.5);
     \draw (2.4,1.5) -- (2.5,1.5);
     \draw (2.5,1.5) .. controls (2.6875,1.5) and (2.6875,1.0834) .. (2.875,1.0834);
     \draw (2.5,1.5) .. controls (3.625,1.5) and (3.625,1.67) .. (4.75,1.67);
     \draw (1.875,2) .. controls (2.6875,2) and (2.6875,1.8334) .. (3.5,1.8334);
     \draw (3.5,1.8334) .. controls (4.125,1.8334) and (4.125,2) .. (4.75,2);
     \draw (1.875,2.125) .. controls (2.375,2.125) and (2.375,2.1666) .. (2.875,2.1666);
     \draw (4.125,0.875) .. controls (4.43875,0.875) and (4.43875,0.66) .. (4.75,0.66);
     \draw (4.125,1) -- (4.75,1);
     \draw (4.125,1.125) .. controls (4.43875,1.125) and (4.43875,1.33) .. (4.75,1.33);
     \draw (4.125,2.25) .. controls (4.43875,2.25) and (4.43875,2.67) .. (4.75,2.67);
     \draw (4.125,2.375) .. controls (4.3125,2.375) and (4.3125,2.75) .. (3.5,2.75);
     \draw (3.5,2.75) .. controls (2.6875,2.75) and (2.6875,2.3334) .. (2.875,2.3334);
     \draw (3.5,2.75) .. controls (0.4375,2.75) and (0.4375,2.0834) .. (0.625,2.0834);
    \end{tikzpicture}
 \end{align}
Here $n=\ul{3}$, and for example $Y_1$ has two input wire and three outputs wires. Each wire in a WD has a specified directionality. As it travels a given wire may split into separate wires, but separate wires cannot come together. The wiring diagram also includes a finite set of delay nodes; in the above case there are four. 

One should think of a wiring diagram $\psi\taking Y_1,\ldots,Y_n\to Z$ as a rule for managing material (or information) flow between the components of an organization. Think of $\psi$ as representing this organization. The individual components of the organization are the interior black boxes (the domain objects of $\psi$) and the exterior black box (the codomain object of $\psi$). Each component supplies material to $\psi$ as well as demands material from $\psi$. For example component $Z$ supplies material on the left side of $\psi$ and demands it on the right side of $\psi$. On the other hand, each $Y_i$ supplies material on its right side and demands material on its left. Like the IDEF0 standard for functional modeling diagrams \cite{NIST}, we always adhere to this directionality.

We insist on one perhaps surprising (though seemingly necessary rule), namely that the wiring diagram cannot connect an output wire of $Z$ directly to an input wire of $Z$. Instead, each output wire of $Z$ is supplied either by an output wire of some $Y(i)$ or by a delay node. 

\begin{announcement}[Morphisms of $\mcW$]\label{ann:morphisms in W}
Let $n\in\Ob(\Fin)$ be a finite set, let $Y\taking n\to\Ob(\mcW)$ be an $n$-indexed set of black boxes, and let $Z\in\Ob(\mcW)$ be another black box. We write
\begin{align}\label{dia:early tensor}
 \inp{Y}&=\amalg_{i\in n}\inp{Y(i)},\\\nonumber
 \outp{Y}&=\amalg_{i\in n}\outp{Y(i)}.
\end{align}
We take $\vset\taking\inp{Y}\amalg\outp{Y}\to\Ob(\Set_*)$ to be the induced map.

A morphism 
$$\psi\taking Y(1),\ldots,Y(n)\to Z$$
in $\mcW_n(Y;Z)$ is called a {\em temporal wiring diagram}, a {\em wiring diagram}, or a {\em WD} for short. It consists of a tuple $(\Del{\psi},\vset,s_\psi)$ as follows.
\ffootnote{-1pt}
{A morphism $\psi\taking Y\to Z$ is in fact an isomorphism class of this data. That is, given two tuples $(\Del{\psi},\vset,s_\psi)$ and $(\Del{\psi}',\vset',s'_\psi)$ as above, with a bijection $\Del{\psi}\iso\Del{\psi}'$ making all the appropriate diagrams commute, we consider these two tuples to constitute the same morphism $\psi\taking Y\to Z$. 
}
\begin{itemize}
 \item $\Del{\psi}\in\Ob(\Fin)$ is a finite set, called the set of {\em delay nodes for $\psi$}. At this point we can define the following sets:
\begin{tabbing}
\hsp\=$\Dem{\psi}:=\outp{Z}\amalg\inp{Y}\amalg\Del{\psi}$\hspace{.5in}\=the set of {\em demand wires in $\psi$}, and\\
\>$\Sup{\psi}:=\inp{Z}\amalg\outp{Y}\amalg\Del{\psi}$\>the set of {\em supply wires in $\psi$}.
\end{tabbing}

 \item $\vset\taking\Del{\psi}\to\Ob(\Set)$ is a function, called the {\em value-set assignment for $\psi$}, such that the diagram 
  $$\xymatrix{
  \Del{\psi}\ar[r]^{\id_{\Del{\psi}}}\ar[d]_{\id_{\Del{\psi}}}\ar[dr]^\vset&\Dem{\psi}\ar[d]^\vset\\
  \Sup{\psi}\ar[r]_\vset&\Set_*
  }
  $$
 commutes (meaning that every delay node demands the same value-set that it supplies).
 \item $s_\psi\taking \Dem{\psi}\to \Sup{\psi}$ is a function, called the {\em supplier assignment for $\psi$}. The supplier assignment $s_\psi$ must satisfy two requirements: 
  \begin{enumerate}
  \item The following diagram commutes:
  $$\xymatrix{\Dem{\psi}\ar[d]_{s_\psi}\ar[dr]^{\vset}\\\Sup{\psi}\ar[r]_{\vset}&\Set_*}$$
  meaning that whenever a demand wire is assigned a supplier, the set of values assigned to these wires must be the same.
  \item If $z\in\outp{Z}$ then $s_\psi(z)\not\in\inp{Z}$. Said another way, 
  $$s_\psi|_{\outp{Z}}\ss\outp{Y}\amalg\Del{\psi},$$
  meaning that a global output cannot be directly supplied by a global input. We call this the {\em non-instantaneity requirement}.
  \end{enumerate} 
\end{itemize}

We have functions $\vset\taking\inp{Z}\amalg\outp{Z}\to\Set_*, \vset\taking\inp{Y(i)}\amalg\outp{Y(i)}\to\Set_*,$ and $\vset\taking\Del{\psi}\to\Set_*$. It should not cause confusion if we use the same symbol to denote the induced functions $\vset\taking \Dem{\psi}\to\Set_*$ and $\vset\taking \Sup{\psi}\to\Set_*$.

\end{announcement}

\begin{remark}\label{rem:tensor}

We have taken the perspective that $\mcW$ is an operad.  One might more naturally think of $\mcW$ as the underlying operad of a symmetric monoidal category whose objects are again black boxes and whose morphisms are again wiring diagrams, though now a morphism connects a single internal domain black box to the external codomain black box.  From this perspective one should merge the many isolated black boxes occurring in the domain of a multicategory wiring diagram into a single black box as the domain of the monoidal category wiring diagram.  
 
Though mathematically equivalent and though we make use of this perspective in the course of our proofs, it is somewhat unnatural to perform this grouping in applications.  For example, though it makes some sense to view ourselves writing this paper and you reading this paper as black boxes inside a single ``information conveying" wiring diagram it would be rather strange to conglomerate all of our collective inputs and outputs so that we become a single meta-information entity.  For reasons of this sort we choose to take the perspective of the underlying operad rather than of a monoidal category.

On the other hand, the notation of monoidal categories is convenient, so we introduce it here. Given a finite set $n$ and an $n$-indexed set of objects $Y\taking n\to\Ob(\mcW)$, we discussed in (\ref{dia:early tensor}) what should be seen as a {\em tensor product} 
$$\bigotimes_{i\in n} Y(i)=(\amalg_{i\in n}\inp{Y(i)},\amalg_{i\in n}\outp{Y(i)},\vset),$$ which we write simply as $Y=(\inp{Y},\outp{Y},\vset)$.

Similarly, given an $n$-indexed set of morphisms $\phi_i\taking X_i\to Y(i)$ in $\mcW$, we can form their tensor product 
$$\bigotimes_{i\in n}\phi_i\taking\bigotimes_{i\in n}X_i\to\bigotimes_{i\in n}Y(i),$$
which we write simply as $\phi\taking X\to Y$, in a similar way. That is, we form a set of delay nodes $\Del{\phi}=\amalg_{i\in n}{\Del{\phi_i}}$, supplies $\Sup{\phi}=\amalg_{i\in n}\Sup{\phi_i}$, demands $\Dem{\phi}=\amalg_{i\in n}\Dem{\phi_i}$, and a supplier assignment $s_\phi=\amalg_{i\in n}s_{\phi_i}$, all by taking the obvious disjoint unions. 

\end{remark}

\begin{example}\label{ex:cookies}

In the example below, we see a big box with three little boxes inside, and we see many wires with arrowheads placed throughout. It is a picture of a wiring diagram $\phi\taking (X_1,X_2,X_3)\to Y$. The big box can be viewed as $Y$, which has some number of input and output wires; however, when we see the big box as a {\em container} of the little boxes wired together, we are actually seeing the morphism $\phi$. 

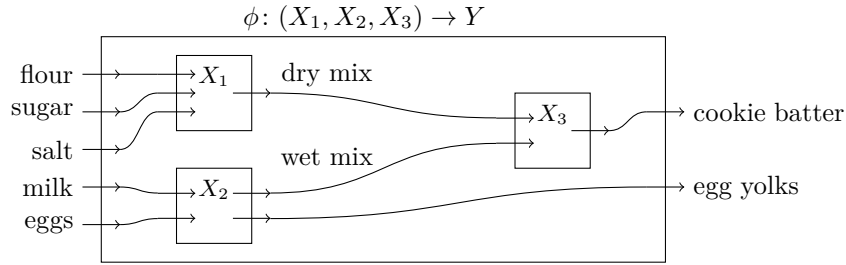
\begin{center}
 \begin{tikzpicture}
 \draw(3.5,3.25) node{$\phi\taking(X_1,X_2,X_3)\to Y$};
  \path (0,0);
  \blackbox{(7.5,3)}{5}{2}{}{0.5}
  \draw (-0.25,0.5) node[left] {eggs};
  \draw (-0.25,1) node[left] {milk};
  \draw (-0.25,1.5) node[left] {salt};
  \draw (-0.25,2) node[left] {sugar};
  \draw (-0.25,2.5) node[left] {flour};
  \draw (3,2.5) node{dry mix};
  \draw (3,1.4) node{wet mix};
  \draw (7.75,1) node[right] {egg yolks};
  \draw (7.75,2) node[right] {cookie batter};
  \path (1,1.75);
  \blackbox{(1,1)}{3}{1}{\small$X_1$}{0.5}
  \path (1,0.25);
  \blackbox{(1,1)}{2}{2}{\small$X_2$}{0.5}
  \path (5.5,1.25);
  \blackbox{(1,1)}{2}{1}{\small$X_3$}{0.5}
  \directarc{(0.25,0.5)}{(0.75,0.58)}
  \directarc{(0.25,1)}{(0.75,0.92)}
  \directarc{(0.25,1.5)}{(0.75,2)}
  \directarc{(0.25,2)}{(0.75,2.25)}
  \directarc{(0.25,2.5)}{(0.75,2.5)}
  \directarc{(2.25,0.58)}{(7.25,1)}
  \directarc{(2.25,0.92)}{(5.25,1.58)}
  \directarc{(2.25,2.25)}{(5.25,1.92)}
  \directarc{(6.75,1.75)}{(7.25,2)}
 \end{tikzpicture}
\end{center}

We aim to explain our terminology of demand and supply, terms which interpret the organization forced on us by the mathematics. Each wire has a demand side and a supply side; when there are no feedback loops, as in the picture above, supplies are on the left side of the wire and demands are to the right, but this is not always the case. Instead, the distinction to make is whether an arrowhead is entering the big box or leaving it: those that enter the big box are supplies to $\phi$, and those that are leaving the big box are demands upon $\phi$. The five left-most arrowheads are entering the big box, so flour, sugar, etc. are being supplied. But flour, sugar, and salt are demands when they leave the big box to enter $X_1$. Counting, one finds 9 supply wires and 9 demand wires (though the equality of these numbers is just a coincidence due to the fact that no wire splits or is wasted).

\end{example}

\subsubsection{Identity morphisms are identity supplier assignments}

Let $Z=(\inp{Z},\outp{Z},\vset)$. The identity wiring diagram $\id_Z\taking Z\to Z$ might be drawn like this:

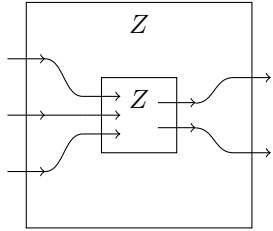
\begin{center}
\begin{tikzpicture}
\path (1,1);
  \blackbox{(1,1)}{3}{2}{$Z$}{0.5}  
\path (0,0);
  \blackbox{(3,3)}{3}{2}{$Z$}{0.5}
\directarc{(.25,.75)}{(0.75,1.25)}
\directarc{(.25,1.5)}{(0.75,1.5)}
\directarc{(.25,2.25)}{(0.75,1.75)}
\directarc{(2.75,1)}{(2.25,1.33)}
\directarc{(2.75,2)}{(2.25,1.66)}  
\end{tikzpicture}
\end{center}
Even though the interior box is of a different size than the exterior box, the way they are wired together is as straightforward as possible.

\begin{announcement}[Identity morphisms in $\mcW$]\label{ann:identity in W}

Let $Z=(\inp{Z},\outp{Z},\vset_Z)$. The identity wiring diagram $\id_Z\taking Z\to Z$ has $\Del{\id_Z}=\emptyset$ with the unique function $\vset\taking\emptyset\to\Ob(\Set)$, so that $\Dem{\id_Z}=\outp{Z}\amalg\inp{Z}$ and $\Sup{\id_Z}=\inp{Z}\amalg\outp{Z}$. The supplier assignment $s_{\id_Z}\taking\Sup{\id_Z}\to\Dem{\id_Z}$ is given by the identity function, which satisfies the non-instantaneity requirement.

\end{announcement}

\subsubsection{Composition of morphisms is achieved by removing intermediary boxes and associated arrow-heads}

We are interested in substituting a wiring diagram into each black box of a wiring diagram, to produce a more detailed wiring diagram. The basic picture to have in mind is the following:

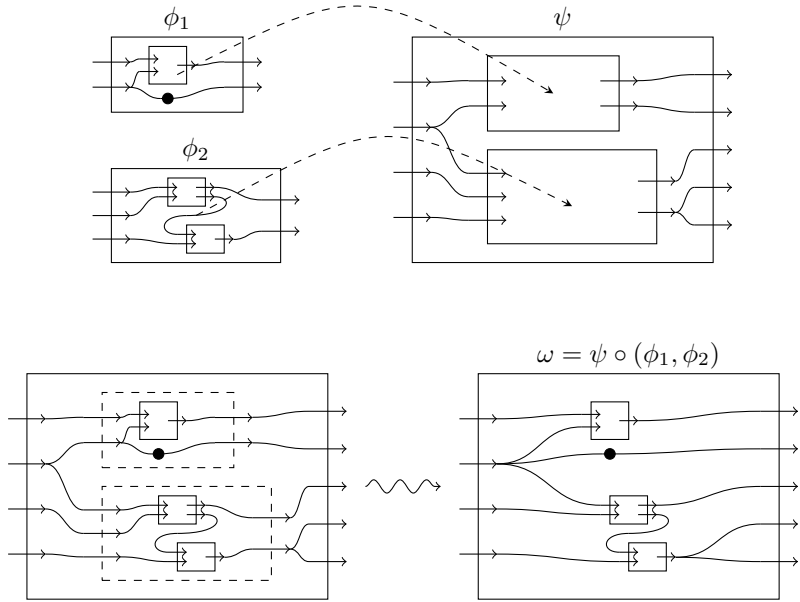
\begin{align*}
 \begin{tikzpicture}
  \draw(1.875,3.25) node{$\phi_1$};
  \path (1,2);
  \blackbox{(1.75,1)}{2}{2}{}{0.5}
  \path (1.5,2.375);
  \blackbox{(0.5,0.5)}{2}{1}{}{0.25}
  \directarc{(1.25,2.33)}{(1.375,2.5416)}
  \directarc{(1.25,2.67)}{(1.375,2.7084)}
  \directarc{(2.125,2.625)}{(2.5,2.67)}
  \filldraw[black] (1.75,2.1825) circle (2pt);
  \directarc{(1.25,2.33)}{(1.65,2.1825)}
  \directarc{(1.65,2.1825)}{(1.7,2.1825)}
  \directarc{(1.75,2.1825)}{(2.5,2.33)}
  \draw[dashed,-stealth] (1.875,2.5) .. controls (4.325,3.75) .. (6.875,2.25);
  \draw[-stealth,dashed] (2.125,0.625) .. controls (4.625,2) .. (7.125,0.75);
  \draw(2.125,1.5) node{$\phi_2$};
  \path (1,0);
  \blackbox{(2.25,1.25)}{3}{2}{}{0.5}
  \path (2,0.125);
  \blackbox{(0.5,0.375)}{2}{1}{}{0.25}
  \directarc{(1.25,0.3125)}{(1.875,0.25)}
  \directarc{(2.625,0.3075)}{(3,0.4166)}
  \loopright{(2.375,0.875)}{(2,0.625)}{7}
  \loopleft{(2,0.625)}{(1.875,0.375)}{7}
  \path (1.75,0.75);
  \blackbox{(0.5,0.375)}{2}{2}{}{0.25}
  \directarc{(1.25,0.625)}{(1.625,0.875)}
  \directarc{(1.25,0.9375)}{(1.625,1)}
  \directarc{(2.375,1)}{(3,0.8333)}
  \draw(7,3.25) node{$\psi$};
  \path (5,0);
  \blackbox{(4,3)}{4}{5}{}{0.5}
  \path (6,1.75);
  \blackbox{(1.75,1)}{2}{2}{}{0.5}
  \directarc{(5.25,1.8)}{(5.75,2.08)}
  \directarc{(5.25,2.4)}{(5.75,2.42)}
  \directarc{(8,2.08)}{(8.75,2)}
  \directarc{(8,2.42)}{(8.75,2.5)}
  \path (6,0.25);
  \blackbox{(2.25,1.25)}{3}{2}{}{0.5}
  \directarc{(5.25,0.6)}{(5.75,0.5625)}
  \directarc{(5.25,1.2)}{(5.75,0.875)}
  \directarc{(5.25,1.8)}{(5.75,1.1875)}
  \directarc{(8.5,0.6666)}{(8.75,0.5)}
  \directarc{(8.5,0.6666)}{(8.75,1)}
  \directarc{(8.5,1.0833)}{(8.75,1.5)}
 \end{tikzpicture}
\end{align*}

\begin{align*}
 \begin{tikzpicture}
  \path (0,0);
  \blackbox{(4,3)}{4}{5}{}{0.5}
  \path (1,1.75);
  \dashblackbox{(1.75,1)}{2}{2}{}{0.5}
  \directarc{(0.25,1.8)}{(0.75,2.08)}
  \directarc{(0.25,2.4)}{(0.75,2.42)}
  \directarc{(3,2.08)}{(3.75,2)}
  \directarc{(3,2.42)}{(3.75,2.5)}
  \path (1.5,2.125);
  \blackbox{(0.5,0.5)}{2}{1}{}{0.25}
  \directarc{(1.25,2.08)}{(1.375,2.2916)}
  \directarc{(1.25,2.42)}{(1.375,2.4584)}
  \directarc{(2.125,2.375)}{(2.5,2.42)}
  \filldraw[black] (1.75,1.9325) circle (2pt);
  \directarc{(1.25,2.08)}{(1.65,1.9325)}
  \directarc{(1.65,1.9325)}{(1.7,1.9325)}
  \directarc{(1.75,1.9325)}{(2.5,2.08)}
  \path (1,0.25);
  \dashblackbox{(2.25,1.25)}{3}{2}{}{0.5}
  \directarc{(0.25,0.6)}{(0.75,0.5625)}
  \directarc{(0.25,1.2)}{(0.75,0.875)}
  \directarc{(0.25,1.8)}{(0.75,1.1875)}
  \directarc{(3.5,0.6666)}{(3.75,0.5)}
  \directarc{(3.5,0.6666)}{(3.75,1)}
  \directarc{(3.5,1.0833)}{(3.75,1.5)}
  \path (2,0.375);
  \blackbox{(0.5,0.375)}{2}{1}{}{0.25}
  \directarc{(1.25,0.5625)}{(1.875,0.5)}
  \directarc{(2.625,0.5575)}{(3,0.6666)}
  \loopright{(2.375,1.125)}{(2,0.875)}{7}
  \loopleft{(2,0.875)}{(1.875,0.625)}{7}
  \path (1.75,1);
  \blackbox{(0.5,0.375)}{2}{2}{}{0.25}
  \directarc{(1.25,0.875)}{(1.625,1.125)}
  \directarc{(1.25,1.1875)}{(1.625,1.25)}
  \directarc{(2.375,1.25)}{(3,1.0833)}
  \draw [->,snake=snake] (4.5,1.5) -- (5.5,1.5);
  \draw(8,3.25) node{$\omega=\psi\circ(\phi_1,\phi_2)$};
  \path (6,0);
  \blackbox{(4,3)}{4}{5}{}{0.5}
  \path (7.5,2.125);
  \blackbox{(0.5,0.5)}{2}{1}{}{0.25}
  \directarc{(6.25,1.8)}{(7.375,2.2916)}
  \directarc{(6.25,2.4)}{(7.375,2.4584)}
  \directarc{(8.125,2.375)}{(9.75,2.5)}
  \filldraw[black] (7.75,1.9325) circle (2pt);
  \directarc{(6.25,1.8)}{(7.65,1.9325)}
  \directarc{(7.65,1.9325)}{(7.7,1.9325)}
  \directarc{(7.75,1.9325)}{(9.75,2)}
  \path (8,0.375);
  \blackbox{(0.5,0.375)}{2}{1}{}{0.25}
  \directarc{(6.25,0.6)}{(7.875,0.5)}
  \directarc{(8.625,0.5575)}{(9.75,0.5)}
  \directarc{(8.625,0.5575)}{(9.75,1)}
  \loopright{(8.375,1.125)}{(8,0.875)}{7}
  \loopleft{(8,0.875)}{(7.875,0.625)}{7}
  \path (7.75,1);
  \blackbox{(0.5,0.375)}{2}{2}{}{0.25}
  \directarc{(6.25,1.2)}{(7.625,1.125)}
  \directarc{(6.25,1.8)}{(7.625,1.25)}
  \directarc{(8.375,1.25)}{(9.75,1.5)}
 \end{tikzpicture}
\end{align*}
On the top we see a wiring diagram $\psi$ in which each internal box, say $Y(1)$ and $Y(2)$, has a corresponding wiring diagram $\phi_1$ and $\phi_2$ respectively. Dropping them into place and then removing the intermediary boxes leaves a single wiring diagram $\omega$. One can see that every input of $Y(i)$ plays a dual role. Indeed, it is a demand from the perspective of $\psi$, and it is a supply from the perspective of $\phi_i$. Similarly, every output of $Y(i)$ plays a dual role as supply in $\psi$ and demand in $\phi_i$. 

In Announcement \ref{ann:composition in W} we will provide the composition formula for $\mcW$. Namely, we will be given morphisms $\phi_i\taking X_i\to Y(i)$ and $\psi\taking Y\to Z$. Each of these has its own delay nodes, $\Del{\phi_i}$ and $\Del{\psi}$ as well as its own supplier assignments. Write $\phi=\bigotimes_i\phi_i\taking X\to Y$ as in Remark \ref{rem:tensor}. For the reader's convenience, we now summarize the demands and supplies for each of the given morphisms $\phi_i\taking X_i\to Y(i)$ and $\psi\taking Y\to Z$, as well as their (not-yet defined) composition $\omega\taking X\to Z$. Let $\Del{\omega}=\Del{\phi}\amalg\Del{\psi}$.
\begin{align}\label{dia:table of notation}
\begin{tabular}{| l || l | l |}
\bhline
\multicolumn{3}{|c|}{Summary of notation for composition in $\mcW$}\\\bhline
{\bf Morphism}&{$\Dem{-}$}&{$\Sup{-}$}\\\bbhline
$\phi_i$&$\outp{Y(i)}\amalg\inp{X_i}\amalg\Del{\phi_i}$&$\inp{Y(i)}\amalg\outp{X_i}\amalg\Del{\phi_i}$\\\hline
$\phi$&$\outp{Y}\amalg\inp{X}\amalg\Del{\phi}$&$\inp{Y}\amalg\outp{X}\amalg\Del{\phi}$\\\hline
$\psi$&$\outp{Z}\amalg\inp{Y}\amalg\Del{\psi}$&$\inp{Z}\amalg\outp{Y}\amalg\Del{\psi}$\\\hline
$\omega$&$\outp{Z}\amalg\inp{X}\amalg\Del{\omega}$&$\inp{Z}\amalg\outp{X}\amalg\Del{\omega}$\\\bhline
\end{tabular}
\end{align}

\begin{lemma}\label{lemma:pushout lemma}

Suppose given morphisms $X\To{\phi} Y$ and $Y\To{\psi}Z$ in $\mcW$, as above. That is, we are given sets of delay nodes, $\Del{\phi}$ and $\Del{\psi}$, as well as supplier assignments 
$$s_{\phi}\taking\Dem{\phi}\to\Sup{\phi}\hsp\tn{and}\hsp s_{\psi}\taking\Dem{\psi}\to\Sup{\psi}$$
each of which is subject to a non-instantaneity requirement,
\begin{align}\label{dia:non-instantaneity requirement}
s_{\phi}\big|_{\outp{Y}}\ss\outp{X}\amalg\Del{\phi}\hsp\tn{and}\hsp s_{\psi}\big|_{\outp{Z}}\ss\outp{Y}\amalg\Del{\psi}.
\end{align}
Let $s_\omega$ be as in Table \ref{dia:table of notation}. It follows that the diagram below is a pushout
$$
\xymatrix@=35pt{
\Sup{\phi}\ar[rr]^{h}&&\Sup{\omega}\urlimit\\
\inp{Y}\amalg\outp{Y}\ar[u]^{g}\ar[rr]_{e}&&\Sup{\psi}\ar[u]_{f}
}
$$
where 
\begin{align}\label{dia:pushout morphisms}
e&=s_\psi\big|_{\inp{Y}}\amalg\id_{\outp{Y}}\\\nonumber
f&=\id_{\inp{Z}}\amalg s_\phi\big|_{\outp{Y}}\amalg\id_{\Del{\psi}}\\\nonumber
g&=\id_{\inp{Y}}\amalg s_{\phi}\big|_{\outp{Y}}\\\nonumber
h&= (f\circ s_\psi)\big|_{\inp{Y}}\amalg\id_{\outp{X}}\amalg\id_{\Del{\phi}}.
\end{align}
Moreover, each of $e,f,g,$ and $h$ commute with the appropriate functions $\vset$.

\end{lemma}

\begin{proof}

We first show that the diagram commutes; here are the calculations on each component:
\begin{align*}
&f\circ e\big|_{\inp{Y}}=f\circ s_\psi\big|_{\inp{Y}}=h\circ g\big|_{\inp{Y}}\\
&f\circ e\big|_{\outp{Y}}=s_\phi\big|_{\outp{Y}}=h\circ g\big|_{\outp{Y}}.
\end{align*}

We now show that the diagram is a pushout. Suppose given a set $Q$ and a commutative solid-arrow diagram (i.e. with $h'\circ g=f'\circ e$):
$$
\xymatrix{
&&&Q\\
\inp{Y}\amalg\outp{X}\amalg\Del{\phi}\ar@/^1pc/[rrru]^{h'}\ar[rr]^{h}&&\inp{Z}\amalg\outp{X}\amalg\Del{\phi}\amalg\Del{\psi}\urlimit\ar@{-->}[ur]_\alpha\\
\inp{Y}\amalg\outp{Y}\ar[u]^{g}\ar[rr]_{e}&&\inp{Z}\amalg\outp{Y}\amalg\Del{\psi}\ar[u]_{f}\ar@/_2pc/[uur]_{f'}
}
$$
Looking at components on which $f$ and $h$ are identities, we see that if we want the equations $\alpha\circ f=f'$ and $\alpha\circ h=h'$ to hold, there is at most one way to define $\alpha\taking\Sup{\omega}\to Q$. Namely,
$$\alpha:=f'\big|_{\inp{Z}\amalg\Del{\psi}}\amalg h'\big|_{\outp{X}\amalg\Del{\phi}}.$$
To see that this definition works, it remains to check that $\alpha\circ f\big|_{\outp{Y}}=f'\big|_{\outp{Y}}$ and that $\alpha\circ h\big|_{\inp{Y}}=h'\big|_{\inp{Y}}$. For the first we use a non-instantaneity requirement (\ref{dia:non-instantaneity requirement}) to calculate:
\begin{align*}
\alpha\circ f\big|_{\outp{Y}}
=\alpha\circ s_\phi\big|_{\outp{Y}}
&=\alpha\big|_{\outp{X}\amalg\Del{\phi}}\circ s_\phi\big|_{\outp{Y}}\\
&=h'\circ s_\phi\big|_{\outp{Y}}\\
&=h'\circ g\big|_{\outp{Y}}=f'\circ e\big|_{\outp{Y}}=f'\big|_{\outp{Y}}
\end{align*}
Now we have shown that $\alpha\circ f=f'$ and the second calculation follows:
\begin{align*}
\alpha\circ h\big|_{\inp{Y}}
=\alpha\circ f\circ s_\psi\big|_{\inp{Y}}
=f'\circ s_\psi\big|_{\inp{Y}}
=f'\circ e\big|_{\inp{Y}}
=h'\circ g\big|_{\inp{Y}}
=h'\big|_{\inp{Y}}
\end{align*}

Each of $e,f,g,h$ commute with the respective functions $\vset$ because each is built solely out of identity functions and supplier assignments. This completes the proof.

\end{proof}

\begin{announcement}[Composition formula for $\mcW$]\label{ann:composition in W}
Let $m,n\in\Ob(\Fin)$ be finite sets and let $t\taking m\to n$ be a function. Let $Z\in\Ob(\mcW)$ be a black box, let $Y\taking n\to\Ob(\mcO)$ be an $n$-indexed set of black boxes, and let $X\taking m\to\Ob(\mcO)$ be an $m$-indexed set of black boxes.  For each element $i\in n$, write $m_i:=t^\m1(i)$ for the pre-image of $i$ under $t$, and write $X_i=X\big|_{m_i}\taking m_i\to\Ob(\mcO)$ for the restriction of $X$ to $m_i$. Then the composition formula 
$$
\circ\taking\mcW_n(Y;Z)\times\prod_{i\in n}\mcW_{m_i}(X_i;Y(i))\too\mcW_{m}(X;Z),
$$
is defined as follows.

Suppose that we are given morphisms $\phi_i\taking X_i\to Y(i)$ for each $i\in n$, which we gather into a morphism $\phi=\bigotimes_i\phi_i\taking X\to Y$ as in Remark \ref{rem:tensor}, and that we are also given a morphism $\psi\taking Y\to Z$. Then we have finite sets of delay nodes $\Del{\phi}$ and $\Del\psi$, and supplier assignments 
$$s_\phi\taking\Dem{\phi}\to\Sup{\phi}\hsp\tn{and}\hsp s_\psi\taking\Dem{\psi}\to\Sup{\psi}$$
as in Announcement \ref{ann:morphisms in W}.

We are tasked with defining a morphism $\omega:=\psi\circ\phi\taking X\to Z$. The set of demand wires and supply wires for $\omega$ are given in Table (\ref{dia:table of notation}). Thus our job is to define a set $\Del{\omega}$ and a supplier assignment $s_\omega\taking\Dem{\omega}\to\Sup{\omega}$.

We put $\Del{\omega}=\Del{\phi}\amalg\Del{\psi}$. It suffices to find a function
$$s_\omega\taking\outp{Z}\amalg\inp{X}\amalg\Del{\omega}\too\inp{Z}\amalg\outp{X}\amalg\Del{\omega},$$
which satisfies the two requirements of being a supplier assignment. We first define the function by making use of the following diagram, where the pushout is as in Lemma \ref{lemma:pushout lemma}:
\begin{align}\label{dia:pushout for supplier assignment}
\xymatrix{
\inp{X}\amalg\Del{\phi}\ar[r]^-{s_\phi\big|_{\inp{X}\amalg\Del{\phi}}}&\Sup{\phi}\ar[r]^h&\Sup{\omega}\urlimit\\
&\inp{Y}\amalg\outp{Y}\ar[u]^g\ar[r]_e&\Sup{\psi}\ar[u]_f\\
&&\outp{Z}\amalg\Del{\psi}\ar[u]_{s_\psi\big|_{\outp{Z}\amalg\Del{\psi}}}
}
\end{align}
Thus we can define a function
\begin{align}\label{dia:composition in W}
s_\omega=h\circ s_\phi\big|_{\inp{X}\amalg\Del{\phi}}\amalg f\circ s_\psi\big|_{\outp{Z}\amalg\Del{\psi}}.
\end{align}

We need to show that $s_\omega$ satisfies the two requirements of being a supplier assignment (see Announcement \ref{ann:morphisms in W}). 
\begin{enumerate}
\item The fact that $s_\omega$ commutes with the appropriate functions $\vset$ follows from the fact that $s_\phi, s_\psi, f,$ and $h$ do so (by Lemma \ref{lemma:pushout lemma}).
\item The fact that the non-instantaneity requirement holds for $s_\omega$, i.e. that $s_\omega(\outp{Z})\ss\outp{X}\amalg\Del{\omega}$, follows from the fact that it holds for $s_\psi$ and $s_\psi$ (see (\ref{dia:non-instantaneity requirement})), as follows.
\begin{align*}
s_\omega(\outp{Z})
&=f\circ s_\psi(\outp{Z})\\
&\ss f(\outp{Y}\amalg\Del{\psi})\\
&=s_\phi(\outp{Y})\amalg\Del{\psi}\\
&\ss\outp{X}\amalg\Del{\phi}\amalg\Del{\psi}=\outp{X}\amalg\Del{\omega}.
\end{align*}
\end{enumerate}

\end{announcement}

\subsection{Running example to ground ideas and notation regarding $\mcW$}\label{sec:ground W}

In this section we will discuss a few objects of $\mcW$ (i.e. black boxes), a couple morphisms of $\mcW$ (i.e. wiring diagrams), and a composition of morphisms. We showed objects and morphisms in more generality above (see Examples \ref{ex:black box} and \ref{ex:cookies}). Here we concentrate on a simple case, which we will take up again in Section \ref{sec:ground P} and which will eventually result in a propagator that outputs the Fibonacci sequence. First, we draw three objects, $X,Y,Z\in\Ob(\mcW)$.
\begin{align}\label{dia:grounding objects}
\begin{tikzpicture}
\path (0,0);
  \blackbox{(1.5,1.5)}{2}{1}{$X$}{.5}
  \draw (-.55,1) node{$a_X$};
  \draw (-.55,.5) node{$b_X$};
  \draw (2.05,.75) node{$c_X$};
\path (4,0);
  \blackbox{(1.5,1.5)}{1}{1}{$Y$}{.5}
  \draw (6.05,.75) node{$c_Y$};
  \draw (3.55,.75) node{$a_Y$};
\path (8,0);
  \blackbox{(1.5,1.5)}{0}{1}{$Z$}{.5}  
    \draw (10.05,.75) node{$c_Z$};
\end{tikzpicture}
\end{align}
These objects are not complete until the pointed sets associated to each wire are specified. Let $N:=(\NN,1)$ be the set of natural numbers with basepoint 1, and put
$$\vset(a_X)=\vset(b_X)=\vset(c_X)=\vset(a_Y)=\vset(c_Y)=\vset(c_Z)=N.$$
Now we draw two morphisms, i.e. wiring diagrams, $\phi\taking X\to Y$ and $\psi\taking Y\to Z$:
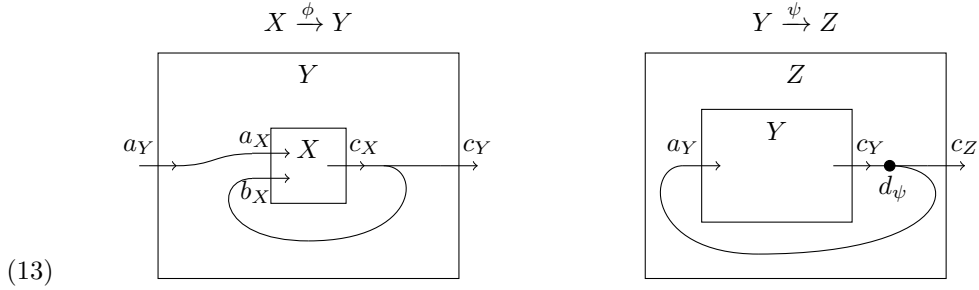
\begin{align}\label{dia:grounding morphisms}
\begin{tikzpicture}
\path (0,0);
  \blackbox{(4,3)}{1}{1}{$Y$}{0.5}
  \draw (-.25,1.75) node{$a_Y$};
  \draw (4.25,1.75) node{$c_Y$};
\path (1.5,1);
  \blackbox{(1,1)}{2}{1}{$X$}{0.5}  
  \draw (1.3,1.85) node{$a_X$};
  \draw (1.3,1.15) node{$b_X$};
  \draw (2.75,1.75) node{$c_X$};
 \directarc{(0.25,1.5)}{(1.25,1.66)}
 \directarc{(2.75,1.5)}{(3.75,1.5)}
 \draw (3.00,1.5) .. controls (3.5,1.5) and (3.5,.5) .. (2,.5);
 \draw (1.25,1.33) .. controls (.75,1.33) and (.75,.5) .. (2,.5);
 \draw (2,3.5) node{$X\To{\phi} Y$};
\end{tikzpicture}
\hspace{.75in}
\begin{tikzpicture}
\path (0,0);
  \blackbox{(4,3)}{0}{1}{$Z$}{0.5}
    \draw (4.25,1.75) node{$c_Z$};
\path (0.75,0.75);
  \blackbox{(2,1.5)}{1}{1}{$Y$}{0.5}  
  \draw (.5,1.75) node{$a_Y$};
  \draw (3.0,1.75) node{$c_Y$};
  \draw (3.3,1.2) node{$d_\psi$};
 \directarc{(3,1.5)}{(3.75,1.5)}
 \filldraw[black] (3.25,1.5) circle (2pt);
 \draw (3.25,1.5) .. controls (4.25,1.5) and (4.25,.325) .. (1.5,.325);
 \draw (.5,1.5) .. controls (0,1.5) and (0,.325) .. (1.5,.325);
 \draw (2,3.5) node{$Y\To{\psi} Z$};
\end{tikzpicture}
\end{align}
To clarify the notion of inputs, outputs, supplies, and demands, we provide two tables that lay out those sets in the case of (\ref{dia:grounding morphisms}). 
\begin{center}
\begin{tabular}{| l || l | l |}
\bhline
\multicolumn{3}{|c|}{Objects shown above}\\\bhline
{\bf Object}&$\inp{-}$&$\outp{-}$\\\bbhline
$X$&$\{a_X,b_X\}$&$\{c_X\}$\\\hline
$Y$&$\{a_Y\}$&$\{c_Y\}$\\\hline
$Z$&$\{\}$&$\{c_Z\}$\\\bhline
\end{tabular}
\hspace{.5in}
\begin{tabular}{| l || l | l | l |}
\bhline
\multicolumn{4}{|c|}{Morphisms shown above}\\\bhline
{\bf Morphism}&$\Del{-}$&$\Dem{-}$&$\Sup{-}$\\\bbhline
$\phi$&$\{\}$&$\{c_Y,a_X,b_X\}$&$\{a_Y,c_X\}$\\\hline
$\psi$&$\{d_\psi\}$&$\{c_Z,a_Y,d_\psi\}$&$\{c_Y,d_\psi\}$\\\bhline
\end{tabular}\\
\end{center}
To specify the morphism $\phi\taking X\to Y$ (respectively $\psi\taking Y\to Z$), we are required not only to provide a set of delay nodes $\Del{\phi}$, which we said was $\Del{\phi}=\emptyset$ (respectively, $\Del{\psi}=\{d_\psi\}$), but also a supplier assignment function $s_\phi\taking\Dem{\phi}\to\Sup{\phi}$ (resp., $s_\psi\taking\Dem{\psi}\to\Sup{\psi}$). Looking at the picture of $\phi$ (resp. $\psi$) above, the reader can trace backward to see how every demand wire is attached to some supply wire. Thus, the supplier assignment $s_\phi$ for $\phi\taking X\to Y$ is
$$c_Y\mapsto c_X,\hsp a_X\mapsto a_Y,\hsp b_X\mapsto c_X,$$
and the supplier assignment $s_\psi$ for $\psi\taking Y\to Z$ is
$$c_Z\mapsto d_\psi,\hsp a_Y\mapsto d_\psi,\hsp d_\psi\mapsto c_Y.$$

We now move on to the composition of $\psi$ and $\phi$. The idea is that we ``plug the $\phi$ diagram into the $Y$-box of the $\psi$ diagram, then erase the $Y$-box". We follow this in two steps below: on the left, we shrink down a copy of $\phi$ and fit it into the $Y$-box of $\psi$. On the right, we erase the $Y$-box:
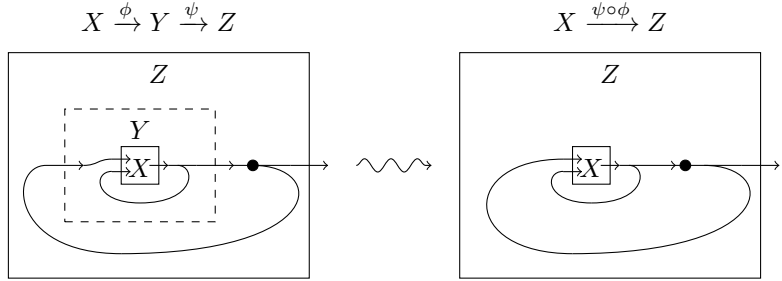
\begin{center}
 \begin{tikzpicture}
  \draw (2,3.5) node{$X\To{\phi} Y\To{\psi} Z$};
  \path (0,0);
  \blackbox{(4,3)}{0}{1}{$Z$}{0.5}
  \path (0.75,0.75);
  \dashblackbox{(2,1.5)}{1}{1}{$Y$}{0.5}  
  \directarc{(3,1.5)}{(3.75,1.5)}
  \filldraw[black] (3.25,1.5) circle (2pt);
  \draw (3.25,1.5) .. controls (4.25,1.5) and (4.25,.325) .. (1.5,.325);
  \draw (.5,1.5) .. controls (0,1.5) and (0,.325) .. (1.5,.325);
  \path (1.5,1.25);
  \blackbox{(0.5,0.5)}{2}{1}{$X$}{0.25}  
  \directarc{(1,1.5)}{(1.375,1.5833)}
  \directarc{(2.125,1.5)}{(2.5,1.5)}
  \draw (2.25,1.5) .. controls (2.5,1.5) and (2.5,1) .. (1.75,1);
  \draw (1.375,1.42) .. controls (1.125,1.42) and (1.125,1) .. (1.75,1);
  \draw [->,snake=snake] (4.625,1.5) -- (5.625,1.5);
  \draw (8,3.5) node{$X\To{\psi\circ\phi} Z$};
  \path (6,0);
  \blackbox{(4,3)}{0}{1}{$Z$}{0.5}
  \directarc{(8.875,1.5)}{(9.75,1.5)}
  \draw [->] (8.125,1.5) -- (8.875,1.5);
  \filldraw[black] (9,1.5) circle (2pt);
  \draw (9.25,1.5) .. controls (10.25,1.5) and (10.25,.325) .. (7.5,.325);
  \draw (7.375,1.5833) .. controls (6,1.5833) and (6,.325) .. (7.5,.325);
  \path (7.5,1.25);
  \blackbox{(0.5,0.5)}{2}{1}{$X$}{0.25}  
  \draw (8.25,1.5) .. controls (8.5,1.5) and (8.5,1) .. (7.75,1);
  \draw (7.375,1.42) .. controls (7.125,1.42) and (7.125,1) .. (7.75,1);
 \end{tikzpicture}
\end{center}
The pushout (\ref{dia:pushout for supplier assignment}) ensures that wires of $Y$ connect wires inside (i.e. from $\phi$) to wires outside (i.e. from $\psi$). In other words, when we erase box $Y$, we do not erase the connections it made for us. We compute the pushout of the diagram 
$$
\{a_Y,c_X\}\From{a_Y\mapsto a_Y,\;\;c_Y\mapsto c_X}\{a_Y,c_Y\}\To{a_Y\mapsto d_\psi,\;\;c_Y\mapsto c_Y}\{c_Y,d_\psi\},
$$
defining $\Sup{\omega}$, to be isomorphic to $\{d_\psi,c_X\}.$ The supplier assignment $s_\omega\taking\Dem{\omega}=\{c_Z,a_X,b_X,d_\psi\}\to\{d_\psi,c_Z\}=\Sup{\omega}$ is given by 
\begin{align}\label{dia:compute composition example}
c_Z\mapsto d_\psi,\hsp a_X\mapsto d_\psi,\hsp b_X\mapsto c_X,\hsp d_\psi\mapsto c_X.
\end{align}
We take this example up again in Section \ref{sec:ground P}, where we show that installing a ``plus" function into box $X$ yields the Fibonacci sequence.

\subsection{Proof that the operad requirements are satisfied by $\mcW$}\label{sec:operad requirements}

We need to show that the announced operad $\mcW$ satisfies the requirements set out by Definition \ref{def:operad}. There are two such requirements: the first says that composing with the identity morphism has no effect, and the second says that composition is associative.

\begin{proposition}

The identity law holds for the announced structure of $\mcW$.

\end{proposition}

\begin{proof}

Let $X_1,\ldots,X_n$ and $Y$ be black boxes and let $\phi\taking X_1,\ldots,X_n\to Y$ be a morphism. We need to show that the following equations hold:
$$\phi\circ(\id_{x_1},\ldots,\id_{x_n})\qeq\phi\hsp\tn{and}\hsp\id_y\circ\phi\qeq\phi.$$

We are given a set $\Del{\phi}$ and a function $\vset\taking\Del{\phi}\to\Ob(\Set)$. Let $\id_X=\bigotimes_{i\in n}\id_{X_i}$, and form $\inp{X}$ and $\outp{X}$ as in Remark \ref{rem:tensor}. Thus we have
$$\Sup{\phi}=\inp{Y}\amalg\outp{X}\amalg\Del{\phi}\hsp\tn{and}\hsp\Dem{\phi}=\outp{Y}\amalg\inp{X}\amalg\Del{\phi}$$
and a supplier assignment $s_\phi\taking\Dem{\phi}\to\Sup{\phi}$.
For each $i\in n$ we have $\Sup{\id_{X_i}}=\Dem{\id_{X_i}}$, and the supplier assignments are the identity, so we have
$$\Sup{\id_{X}}=\Dem{\id_{X}}=\inp{X}\amalg\outp{X}$$ 
The supplier assignment $s_{\id_{X}}$ is the identity function. Similarly, $\Sup{\id_{Y}}=\Dem{\id_Y}=\inp{Y}\amalg\outp{Y}$, and the supplier assignment $s_{\id_{Y}}$ is the identity function.

Let $\omega=\phi\circ(\id_{X_1},\ldots,\id_{X_n})$ and $\omega'=\id_Y\circ\phi$. Then the relevant pushouts become
$$
\xymatrix@=30pt{
\inp{X}\ar[r]^-{\id\big|_{\inp{X}}}&\Sup{\id_X}\ar[r]^{s_\phi\big|{\inp{X}}\amalg\id\big|_{\outp{X}}}&\Sup{\omega}\urlimit\\
&\inp{X}\amalg\outp{X}\ar@{=}[u]\ar[r]&\Sup{\phi}\ar[u]\\
&&\outp{Y}\amalg\Del{\phi}\ar[u]_{s_\phi\big|_{\outp{Y}\amalg\Del{\phi}}}
}
$$
$$
\xymatrix@=30pt{
\inp{X}\amalg\Del{\phi}\ar[r]^{s_\phi\big|_{\inp{X}\amalg\Del{\phi}}}&\Sup{\phi}\ar[r]&\Sup{\omega'}\urlimit\\
&\inp{Y}\amalg\outp{Y}\ar@{=}[r]\ar[u]&\Sup{\id_Y}\ar[u]_{\id\big|_{\inp{Y}}\amalg s_\phi\big|_{\outp{Y}}}\\
&&\outp{Y}\ar[u]_{\id\big|_{\outp{Y}}}
}
$$
The pushout of an isomorphism is an isomorphism so we have isomorphisms $\Sup{\phi}\iso\Sup{\omega}$ and $\Sup{\phi}\iso\Sup{\omega'}$. 
\ffootnote{-3pt}
{Note that a morphism (e.g. $\omega$) in $\mcW$ are defined only up to isomorphism class of tuples $(\Del{\omega},\vset,s_\omega)$, see Announcement \ref{ann:morphisms in W}.} 
In both the case of $\omega$ and $\omega'$, one checks using (\ref{dia:pushout morphisms}) that the induced supplier assignments are also in agreement (up to isomorphism), $s_{\omega}=s_\phi=s_{\omega'}.$

\end{proof}

\begin{proposition}

The associativity law holds for the announced structure of $\mcW$.

\end{proposition}

\begin{proof}

Suppose we are given morphisms $\tau\taking W\to X$, $\phi\taking X\to Y$ and $\psi\taking Y\to Z$. We must check that $(\psi\circ\phi)\circ\tau=\psi\circ(\phi\circ\tau)$. With notation as in Lemma \ref{lemma:pushout lemma}, pushout square defining $\phi\circ\tau$ and then $\psi\circ(\phi\circ\tau)$ are these:
$$
\xymatrix{
\Sup{\tau}\ar[r]^{h_{\phi,\tau}}&\Sup{\phi\circ\tau}\urlimit\\
\inp{X}\amalg\outp{X}\ar[u]^{g_{\phi,\tau}}\ar[r]_-{e_{\phi,\tau}}&\Sup{\phi}\ar[u]_{f_{\phi,\tau}}
}
\hspace{.7in}
\xymatrix{
\Sup{\phi\circ\tau}\ar[r]^{h_{\psi,\phi\circ\tau}}&\Sup{\psi\circ(\phi\circ\tau)}\urlimit\\
\inp{Y}\amalg\outp{Y}\ar[r]_-{e_{\psi,\phi\circ\tau}}\ar[u]^{g_{\psi,\phi\circ\tau}}&\Sup{\psi}\ar[u]_{f_{\psi,\phi\circ\tau}}
}
$$
whereas the pushout square defining $\psi\circ\phi$ and then $(\psi\circ\phi)\circ\tau$ are these:
$$
\xymatrix{
\Sup{\phi}\ar[r]^{h_{\psi,\phi}}&\Sup{\psi\circ\phi}\urlimit\\
\inp{Y}\amalg\outp{Y}\ar[r]_-{e_{\psi,\phi}}\ar[u]^{g_{\psi,\phi}}&\Sup{\psi}\ar[u]_{f_{\psi,\phi}}
}
\hspace{.7in}
\xymatrix{
\Sup{\tau}\ar[r]^{h_{\psi\circ\phi,\tau}}&\Sup{(\psi\circ\phi)\circ\tau}\urlimit\\
\inp{X}\amalg\outp{X}\ar[u]^{g_{\psi\circ\phi,\tau}}\ar[r]_-{e_{\psi\circ\phi,\tau}}&\Sup{\psi\circ\phi}\ar[u]_{f_{\psi\circ\phi,\tau}}
}
$$
One checks directly from the formulas (\ref{dia:pushout morphisms}) that $e_{\psi\circ\phi,\tau}=h_{\psi,\phi}\circ e_{\phi,\tau}$ as functions $\inp{X}\amalg\outp{X}\to\Sup{\psi\circ\phi}$, and that $g_{\psi,\phi\circ\tau}=f_{\phi,\tau}\circ g_{\psi,\phi}$ as functions $\inp{Y}\amalg\outp{Y}\to\Sup{\phi\circ\tau}$.

We combine them into the following pushout diagram:
$$\xymatrix{
\Sup{\tau}\ar[r]^{h_{\phi,\tau}}&\Sup{\phi\circ\tau}\ar[r]^{h_{\psi,\phi\circ\tau}}\urlimit&\Sup{\psi\circ\phi\circ\tau}\urlimit\\
\inp{X}\amalg\outp{X}\ar[u]^{g_{\phi,\tau}}\ar[r]_-{e_{\phi,\tau}}&\Sup{\phi}\ar[u]^{f_{\phi,\tau}}\ar[r]_{h_{\psi,\phi}}&\Sup{\psi\circ\phi}\ar[u]_{f_{\psi\circ\phi,\tau}}\urlimit\\
&\inp{Y}\amalg\outp{Y}\ar[r]_-{e_{\psi,\phi}}\ar[u]^{g_{\psi,\phi}}&\Sup{\psi}\ar[u]_{f_{\psi,\phi}}
}
$$
The pasting lemma for pushout squares ensures that the set labeled $\Sup{\psi\circ\phi\circ\tau}$ is isomorphic to $\Sup{\psi\circ(\phi\circ\tau)}$ and to $\Sup{(\psi\circ\phi)\circ\tau}$, so these are indeed isomorphic to each other. It is also easy to check using the formulas provided in (\ref{dia:composition in W}) and (\ref{dia:pushout morphisms}) that the supplier assignments 
$$\Dem{\psi\circ\phi\circ\tau}=\outp{Z}\amalg\inp{W}\amalg\Del{\tau}\amalg\Del{\phi}\amalg\Del{\psi}\too\Sup{\psi\circ\phi\circ\tau}$$ 
agree regardless of the order of composition. This proves the result.

\end{proof}

\section{$\mcP$, the algebra of propagators on $\mcW$}\label{sec:algebra}

In this section we will introduce our algebra of propagators on $\mcW$. This is where form meets function: the form called ``black box" is a placeholder for a propagator, i.e. a function, that carries input streams to output streams, and the form called ``wiring diagram" is a placeholder for a circuit that links propagators together to form a larger propagator. 

To formalize these ideas we introduce the mathematical notion of operad algebra in Section \ref{sec:algebra def}. In Section \ref{sec:lists streams props} we discuss some preliminaries on lists and streams, and define our notion of historical propagator. In Section \ref{sec:algebra announcements} we announce our algebra of these propagators and in Section \ref{sec:ground P} we ground it in our running example. Finally in Section \ref{sec:algebra requirements} we prove that the announced structure really satisfies the requirements of being an algebra.

\subsection{Definition and basic examples of algebras}\label{sec:algebra def}

In this section we give the formal definition for algebras over an operad. 

\begin{definition}

Let $\mcO$ be an operad. An {\em $\mcO$-algebra}, denoted $F\taking\mcO\to\Sets$, is defined as follows: One announces some constituents (A. map on objects, B. map on morphisms) and proves that they satisfy some requirements (1. identity law, 2. composition law). Specifically,
\begin{enumerate}[\hsp A.]
\item one announces a function $\Ob(F)\taking\Ob(\mcO)\to\Ob(\Sets)$.
\item for each object $y\in\Ob(\mcO)$, finite set $n\in\Ob(\Fin)$, and $n$-indexed set of objects $x\taking n\to\Ob(\mcO)$, one announces a function 
$$F_n\taking\mcO_n(x;y)\to\Hom_\Sets(Fx;Fy).$$
\end{enumerate}
As in B. above, we often denote $\Ob(F)$, and also each $F_n$, simply by $F$. 

These constituents (A,B) must satisfy the following requirements:
\begin{enumerate}[\hsp 1.]
\item For each object $x\in\Ob(\mcO)$, the equation $F(\id_x)=\id_{Fx}$ holds.
\item Let $s\taking m\to n$ be a morphism in $\Fin$. Let $z\in\Ob(\mcO)$ be an object, let $y\taking n\to\Ob(\mcO)$ be an $n$-indexed set of objects, and let $x\taking m\to\Ob(\mcO)$ be an $m$-indexed set of objects. Then, with notation as in Definition \ref{def:operad}, the following diagram of sets commutes:
\begin{align}\label{dia:operad functor on composition}
\xymatrix{
\mcO_n(y;z)\times\prod_{i\in n}\mcO_{m_i}(x_i;y(i))\ar[d]_-F\ar[r]^-\circ
&
\mcO_m(x;z)\ar[d]^F\\
\Hom_\Sets(Fy;Fz)\times\prod_{i\in n}\Hom_\Sets(Fx_i;Fy(i))\ar[r]_-\circ
&
\Hom_\Sets(Fx;Fz)
}
\end{align}
\end{enumerate}

\end{definition}

\begin{example}\label{ex:commutative monoids}

Let $\mcE$ be the commutative operad of Example \ref{ex:commutative operad}. An $\mcE$-algebra $S\taking\mcE\to\Sets$ consists of a set $M\in\Ob(\Set)$, and for each natural number $n\in\NN$ a morphism $\mu_n\taking M^n\to M$. It is not hard to see that, together, the morphism $\mu_2\taking M\times M\to M$ and the element $\mu_0\taking\singleton\to M$ give $M$ the structure of a commutative monoid. Indeed, the associativity and unit axioms are encoded in the axioms for operads and their morphisms. The commutativity of multiplication arises by applying the commutative diagram (\ref{dia:operad functor on composition}) in the case $s\taking \{1,2\}\to \{1,2\}$ is the non-identity bijection, as discussed in Remark \ref{rem:symmetry}.

\end{example}

\subsection{Lists, streams, and historical propagators}\label{sec:lists streams props}

In this section we discuss some background on lists. We also develop our notion of historical propagator, which formalizes the idea that a machine's output at time $t_0$ can depend only on what has happened previously, i.e. for time $t<t_0$. While strictly not necessary for the development of this paper, we also discuss the relation of historical propagators to streams.

Given a set $S$, an {\em $S$-list} is a pair $(t,\ell)$, where $t\in\NN$ is a natural number and $\ell\taking\{1,2,\ldots,t\}\to S$ is a function. We denote the set of $S$-lists by $\List(S)$. We call $t$ the {\em length} of the list; in particular a list may be empty because we may have $t=0$. Note that there is a canonical bijection
$$\List(S)\iso\coprod_{t\in\NN}S^t.$$
We sometimes denote a list simply by $\ell$ and write $|\ell|$ to denote its length; that is we have the component projection $|\cdot|\taking\List(A)\to\NN$. We typically write-out an $S$-list as $\ell=[\ell(1),\ell(2),\ldots,\ell(t)]$, where each $\ell(i)\in S$. We denote the empty list by $[\;]$. Given a function $f\taking S\to S'$, there is an induced function $\List(f)\taking\List(S)\to\List(S')$ sending $(t,\ell)$ to $(t,f\circ\ell)$; in the parlance of computer science $\List(f)$ is the function that ``maps $f$ over $\ell$". 

Given sets $X_1,\ldots,X_k\in\Ob(\Set)$, an element in $\List(\prod_{1\leq i\leq k}X_i)$ is a list of $k$-tuples. Given sets $A$ and $B$ there is a bijection
$$\zipwith\taking\List(A)\times_\NN\List(B)\Too{\iso}\List(A\times B),$$
where on the left we have formed the fiber product of the diagram $\List(A)\To{|\cdot|}\NN\From{|\cdot|}\List(B)$. We call this bijection {\em zipwith}, following the terminology from modern functional programming languages. The idea is that an $A$-list $\ell_A$ can be combined with a $B$-list $\ell_B$, as long as they have the same length $|\ell_A|=|\ell_B|$; the result will be an $(A\times B)$-list $\ell_A\zipwith\ell_B$ again of the same length.  We will usually abuse this distinction and freely identify $\List(A\times B)\iso\List(A)\times_\NN\List(B)$ with its image in $\List(A)\times\List(B)$. For example, we may consider the $\NN\times\NN$-list
\[[(1,2),(3,4),(5,6)]=[1,3,5]\zipwith[2,4,6]\]
as an element of $\List(\NN)\times\List(\NN)$. Hopefully this will not cause confusion.

Let $\List_{\geq 1}(S)\ss\List(S)$ denote the set $\amalg_{t\geq 1}S^t$. We write $\partial_S\taking\List_{\geq1}(S)\to\List(S)$ to denote the function that drops off the last entry. More precisely, for any integer $t\geq 1$ if we consider $\ell$ as a function $\ell\taking\{1,2,\ldots,t\}\to S$, then the list $\partial_S\ell$ is given by pre-composition with the subset consisting of the first $t-1$ elements, 
$$\{1,2,\ldots,t-1\}\inj\{1,2,\ldots,t\}\To{\ell}S.$$ 
For example we have $\partial[0,1,4,9,16]=[0,1,4,9]$.

\begin{definition}\label{def:historical}

Let $R,S$ be pointed sets and let $n\in\NN$. A {\em $n$-historical propagator $f$ from $R$ to $S$} is a function $f\taking\List(R)\to\List(S)$ satisfying the following conditions:
\begin{enumerate}
\item If a list $\ell\in\List(R)$ has length $|\ell|=t$, then $|f(\ell)|=t+n$, 
\item If $\ell\in\List(R)$ is a list of length $t\geq 1$, then 
$$\partial_Sf(\ell)=f(\partial_R\ell).$$
\end{enumerate}
We denote the set of $n$-historical propagators from $R$ to $S$ by $\Hist^n(R,S)$. If $f$ is $n$-historical for some $n\geq 0$ we say that $f$ is {\em historical}.

\end{definition}

We usually drop the subscript from the symbol $\partial_-$, writing e.g. $\partial f(\ell)=f(\partial\ell)$.

\begin{example}\label{ex:delay}

Let $S$ be a pointed set and let $n\in\NN$ be a natural number. Define an $n$-historical propagator $\delta^n\in\Hist^n(S,S)$ as follows for $\ell\in\List(S)$: 
$$\delta^n(\ell)(i)=
\begin{cases}
*&\tn{ if }1\leq i\leq n\\
\ell(i-n)&\tn{ if }n+1\leq i\leq t+n
\end{cases}
$$
We call $\delta^n$ the {\em $n$-moment delay function}. For example if $n=3, S=\{a,b,c,d\}\amalg\{*\}$, and $\ell=[a,a,b,*,d]\in S^5$ then $\delta^3(S)=[*,*,*,a,a,b,*,d]\in S^8$.

\end{example}

The following Lemma describes the behavior of historical functions.

\begin{lemma}\label{lemma:little facts on historicality}

Let $S,S',S'',T,T'\in\Set_*$ be pointed sets.
\begin{enumerate}
\item Let $f\taking S\to T$ be a function. The induced function $\List(f)\taking\List(S)\to\List(T)$ is $0$-historical.
\item Given $n$-historical propagators $q\in\Hist^n(S,S')$ and $r\in\Hist^n(T,T')$, there is an induced $n$-historical propagator $q\times r\in\Hist^n(S\times T,S'\times T')$.
\item Given $q\in\Hist^m(S,S')$ and $q'\in\Hist^{n}(S',S'')$, then $q'\circ q\taking\List(S)\to\List(S'')$ is $(m+n)$-historical.
\item If $n\geq 1$ is an integer and $q\in\Hist^n(S,S')$ is $n$-historical then $\partial q\taking\List(S)\to\List(S')$ is $(n-1)$-historical.
\end{enumerate}

\end{lemma}

\begin{proof}
We show each in turn.
\begin{enumerate}
\item Let $\ell\in\List(S)$ be a list of length $t$. Clearly, $\List(f)$ sends $\ell$ to a list of length $t$. If $t\geq 1$ then the fact that $\partial\List(f)(\ell)=\List(f)(\partial\ell)$ follows by associativity of composition in $\Set$. That is, $\List(f)(\ell)$ is the right-hand composition and $\partial\ell$ is the left-hand composition below:
$$\{1,\ldots,t-1\}\inj\{1,\ldots,t\}\To{\ell}S\To{f}T.$$
\item On the length $t$ component we use the function $(S\times T)^t=S^t\times T^t\To{q\times r}S^{t+n}\times T^{t+n}=(S\times T)^{t+n}$. As necessary, we have 
$$\partial\circ(q\times r)=\partial q\times\partial r=q\partial\times r\partial=(q\times r)\circ\partial.$$
\item This is straightforward; for example the second condition is checked
$$\partial q'(q(\ell))=q'(\partial q(\ell))=q'(q(\partial\ell)).$$
\item On lengths we indeed have $|\partial q(\ell)|=|q(\ell)|-1=|\ell|+n-1$. If $|\ell|=t\geq 1$ then $\partial(\partial q)(\ell)=\partial(\partial q(\ell))=\partial q(\partial\ell)$ because $q$ is historical.
\end{enumerate}
\end{proof}

\begin{definition}

Let $S$ be a pointed set. An {\em $S$-stream} is a function $\sigma\taking\NN_{\geq 1}\to S$. We denote the set of $S$-streams by $\Strm{S}$. 

For any natural number $t\in\NN$, let $\sigma\strst{t}\in\List(S)$ denote the list of length $t$ corresponding to the composite $\{1,2,\ldots,t\}\inj\NN_{\geq 1}\To{\sigma}S$ and call it {\em the $t$-restriction of $S$}. 

\end{definition}

\begin{lemma}

Let $S$ be a pointed set, let $\{*\}$ be a pointed set with one element, and let $n\in\NN$ be a natural number. There is a bijection 
$$\Hist^n(\{*\},S)\To{\iso}\Strm{S}.$$

\end{lemma}
 
\begin{proof}

For any natural number $t\in\NN$, let $\disc{t}=\{1,2,\ldots,t\}\in\Ob(\Set)$. Let $[\NN]$ be the poset (considered as a category) with objects $\{\disc{t}\|t\in\NN\}$, ordered by inclusion of subsets. For any $n\in\NN$ there is a functor $[\NN]\to\Set$ sending $\ul{t}\in\Ob([\NN])$ to $\{1,2,\ldots,t+n\}\in\Ob(\Set)$.

For any $n\in\NN$, there is a bijection $\NN\iso\colim_{t\in[\NN]}\{1,2,\ldots,t+n\}.$
Thus we have a bijection
$$\Strm{S}=\Hom_\Set(\NN_{\geq 1},S)\iso\lim_{t\in[\NN]}\Hom_\Set(\{1,2,\ldots,t+n\},S).$$ 

On the other hand, an $n$-historical function $f\taking\List(\{*\})\to\List(S)$ acts as follows. For each $t\in\NN$ and list $[*,\ldots,*_t]$ of length $t$, it assigns a list $f([*,\ldots,*_t])\in\List(S)$ of length $t+n$, i.e. a function $\{1,\ldots,t+n\}\to S$, such that $f([*,\ldots,*_{t-1}])$ is the restriction to the subset $\{1,\ldots,t+n-1\}$. 

The fact that these notions agree follows from the construction of limits in the category $\Set$.

\end{proof}

Below we define an awkward-sounding notion of {\em $n$-historical stream propagator}. The idea is that a function carrying streams to streams is $n$-historical if, for all $t\in\NN$, its output up to time $t+n$ depends only on its input up to time $t$. In Proposition \ref{prop:historicality aligns} we show that this notion of historicality for streams is equivalent to the notion for lists given in Definition \ref{def:historical}.

\begin{definition}

Let $S$ and $T$ be pointed sets, and let $n\in\NN$ be a natural number. A function $f\taking\Strm{S}\to\Strm{T}$ is called an {\em $n$-historical stream propagator} if, given any natural number $t\in\NN$ and any two streams $\sigma,\sigma'\in\Strm{S}$, if $\sigma\strst{t}=\sigma'\strst{t}$ then $f(\sigma)\strst{t+n}=f(\sigma')\strst{t+n}$. Let $\Hist_{strm}^n(S,T)$ denote the set of $n$-historical stream propagators $\Strm{S}\to\Strm{T}$.

\end{definition}

\begin{proposition}\label{prop:historicality aligns}

Let $S$ and $T$ be pointed sets. There is a bijection 
$$\Hist^n(S,T)\To{\iso}\Hist^n_{strm}(S,T).$$

\end{proposition}

\begin{proof}

We construct two functions $\alpha\taking\Hist^n(S,T)\to\Hist^n_{strm}(S,T)$ and $\beta\taking\Hist^n_{strm}(S,T)\to\Hist^n(S,T)$ that are mutually inverse.

Given an $n$-historical function $f\taking\List(S)\to\List(T)$ and a stream $\sigma\in\Strm{S}$, define the stream $\alpha(f)(\sigma)\taking\NN_{\geq 1}\to T$ to be the function whose $(t+n)$-restriction (for any $t\in\NN$) is given by
$$\alpha(f)(\sigma)\strst{t+n}=f(\sigma\strst{t}).$$
Because $f$ is historical, this construction is well defined.

Given an $n$-historical stream propagator $F\taking\Strm{S}\to\Strm{T}$ and a list $\ell\in\List(S)$ of length $|\ell|=t$, let $\ell_*\in\Strm{S}$ denote the stream $\NN_{\geq 1}\to S$ given on $i\in\NN_{\geq 1}$ by 
$$\ell_*(i)=
\begin{cases}
\ell(i)&\tn{ if }1\leq i\leq t\\
*&\tn{ if }i\geq t+1.
\end{cases}
$$
Now define the list $\beta(F)(\ell)\in\List(T)$ by 
$$\beta(F)(\ell)=F(\ell_*)\strst{t+n}.$$
One checks directly that for all $F\in\Hist_{strm}^n(S,T)$ we have $\alpha\circ\beta(F)=F$ and that for all $f\in\Hist^n(S,T)$ we have $\beta\circ\alpha(f)=f$.

\end{proof}

The above work shows that the notion of historical propagator is the same whether one considers it as acting on lists or on streams. Throughout the rest of this paper we work solely  with the list version. However, we sometimes say the word ``stream" (e.g. ``a propagator takes a stream of inputs and returns a stream of outputs") for the image it evokes.
 
 \subsection{The announced structure of the propagator algebra $\mcP$}\label{sec:algebra announcements}

In this section we will announce the structure of our $\mcW$-algebra of propagators, which we call $\mcP$. That is, we must specify
\begin{itemize}
 \item the set $\mcP(Y)$ of allowable ``fillers" for each black box $Y\in\Ob(\mcW)$,
 \item how a wiring diagram $\psi\taking Y_1,\ldots,Y_n\to Z$ and a filler for each $Y_i$ serves to produce a filler for $Z$.
\end{itemize}
In this section we will explain in words and then formally announce mathematical definitions.  In Section \ref{sec:operad requirements} we will prove that the announced structure has the required properties.

As mentioned above, the idea is that each black box is a placeholder for (i.e. can be filled with) those propagators which carry the specified local input streams to the specified local output streams. Each wiring diagram with propagators installed in each interior black box will constitute a new propagator for the exterior black box, which carries the specified global input streams to the specified global output streams. We now go into more detail and make these ideas precise.

\subsubsection{Black boxes are filled by historical propagators}

Let $Z=(\inp{Z},\outp{Z},\vset)$ be an object in $\mcW$. Recall that each element $w\in\inp{Z}$ is called an input wire, which carries a set $\vset(w)$ of possible values, and that element $w'\in\outp{Z}$ is called an output wire, which also carries a set $\vset(w')$ of possible values. This terminology is suggestive of a machine, which we call a {\em historical propagator} (or {\em propagator} for short), which takes a list of values on each input wire, processes it somehow, and emits a list of values on each output wire. The propagator's output at time $t_0$ can depend on the input it received for time $t<t_0$, but not on input that arrives later.

\begin{announcement}[$\mcP$ on objects]\label{ann:P on objects}

Let $Z=(\inp{Z},\outp{Z},\vset)$ be an object in $\mcW$. For any subset $I\ss\inp{Z}\amalg\outp{Z}$ we define 
$$\vset_I=\prod_{i\in I}\vset(i).$$
In particular, if $I=\emptyset$ then $\vset_I$ is a one-element set.

We define $\mcP(Z)\in\Ob(\Set)$ to be the set of {\em 1-historical propagators of type $Z$},
$$\mcP(Z):=\Hist^1(\vset_{\inp{Z}},\vset_{\outp{Z}}).$$

\end{announcement}

Consider the propagator below, which has one input wire and one output wire, say both carrying integers.
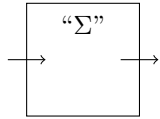
\begin{center}
\begin{tikzpicture}
\path(0,0);
\blackbox{(1.5,1.5)}{1}{1}{$``\Sigma"$}{.5};
\end{tikzpicture}
\end{center}
The name $``\Sigma"$ suggests that this propagator takes a list of integers and returns their running total. But for it to be 1-historical, its input up to time $t$ determines its output up to time $t+1$. Thus for example it might send an input list $\ell:=[1,3,5,7,10]$ of length $5$ to the output list $``\Sigma"(\ell)=[0,1,4,9,16,26]$ of length $6$.

\begin{remark}\label{rem:tensor2}

As in Remark \ref{rem:tensor} the following notation is convenient. Given a finite set $n\in\Ob(\Fin)$ and black boxes $Y_i\in\Ob(\mcW)$ for $i\in n$, we can form $Y=\bigotimes_{i\in n}Y_i$, with for example $\inp{Y}=\amalg_{i\in n}\inp{Y_i}$. Similarly, given a $1$-historical propagator $g_i\in\mcP(Y_i)$ for each $i\in n$ we can form a 1-historical propagator $g:=\bigotimes_{i\in n}g_i\in\Hist^1(\vset_{\inp{Y}},\vset_{\outp{Y}})$ simply by $g=\prod_{i\in n}g_i$.

\end{remark}

\subsubsection{Wiring diagrams shuttle value streams between propagators}

Let $Z\in\Ob(\mcW)$ be a black box, let $n\in\Ob(\Fin)$ be a finite set, and let $Y\taking n\to\Ob(\mcW)$ be an $n$-indexed set of black boxes. A morphism $\psi\taking Y\to Z$ in $\mcW$ is little more than a supplier assignment $s_\psi\taking\Dem{\psi}\to\Sup{\psi}$. In other words, it connects each demand wire to a supply wire carrying the same set of values. Therefore, if a propagator is installed in each black box $Y(i)$, then $\psi$ tells us how to take each value stream being produced by some propagator and feed it into the various propagators that it supplies.

\begin{announcement}[$\mcP$ on morphisms]\label{ann:P on morphisms}

Let $Z\in\Ob(\mcW)$ be a black box, let $n\in\Ob(\Fin)$ be a finite set, let $Y\taking n\to\Ob(\mcW)$ be an $n$-indexed set of black boxes, and let $\psi\taking Y\to Z$ be a morphism in $\mcW$. We must construct a function
$$\mcP(\psi)\taking\mcP(Y(1))\times\cdots\times\mcP(Y(n))\to\mcP(Z).$$
That is, given a historical propagator $g_i\in\Hist^1(\vset_{\inp{Y(i)}},\vset_{\outp{Y(i)}})$ for each $i\in n$, we need to produce a historical propagator $\mcP(\psi)(g_1,\ldots,g_n)\in\Hist^1(\vset_{\inp{Z}},\vset_{\outp{Z}}).$ Define $g\in\Hist^1(\vset_{\inp{Y}},\vset_{\outp{Y}})$ by $g:=\bigotimes_{i\in n}g_i$, as in Remark \ref{rem:tensor2}. Let $\inDem{\psi}=\inp{Y}\amalg\Del{\psi}$ and $\inSup{\psi}=\outp{Y}\amalg\Del{\psi},$ denote the set of internal demands of $\psi$ and the set of internal supplies of $\psi$, respectively.
 
We will define $\mcP(\psi)(g)$ by way of five helper functions:
\begin{align*}
S_\psi&\in\Hist^0(\vset_{\Sup{\psi}},\vset_{\Dem{\psi}}),\\
S'_\psi&\in\Hist^0(\vset_{\Sup{\psi}},\vset_{\inDem{\psi}}),\\
S''_\psi&\in\Hist^0(\vset_{\inSup{\psi}},\vset_{\outp{Z}}),\\
E_{\psi,g}&\in\Hist^1(\vset_{\inDem{\psi}},\vset_{\inSup{\psi}}),\\
C_{\psi,g}&\in\Hist^0(\vset_{\inp{Z}},\vset_{\Sup{\psi}}),
\end{align*}
where we will refer to the $S_\psi,S'_\psi,S''_\psi$ as ``shuttle", $E_{\psi,g}$ as ``evaluate", and $C_{\psi,g}$ as ``cascade".  We will abbreviate by $\vinp{Z}$ the set $\List(\vset_{\inp{Z}})$, and similarly for $\vSup{\psi},$   $\vinDem{\psi},$ etc.

By Announcement \ref{ann:morphisms in W}, a morphism $\psi\taking Y\to Z$ in $\mcW$ is given by a tuple $(\Del{\psi},\vset,s_\psi)$, where in particular we remind the reader of a commutative diagram
$$\xymatrix{\Dem{\psi}\ar[d]_{s_\psi}\ar[dr]^{\vset}\\\Sup{\psi}\ar[r]_{\vset}&\Set_*}$$
where we require $s_\psi(\outp{Z})\ss\inSup{\psi}$.  The function $s_\psi\taking\Dem{\psi}\to\Sup{\psi}$ induces the coordinate projection function $\pi_{s_\psi}\taking\vset_{\Sup{\psi}}\to\vset_{\Dem{\psi}}$ (see Section \ref{sec:notation}). Applying the functor $\List$ gives a $0$-historical function (see Lemma \ref{lemma:little facts on historicality}), $\List(\pi_{s_\psi})$ which we abbreviate as 
$$S_\psi\taking\vSup{\psi}\to\vDem{\psi}.$$
This is the function that shuttles a list of tuples from where they are supplied directly along a wire to where they are demanded. We define a commonly-used projection, 
$$S'_\psi:=\pi_{\vinDem{\psi}}\circ S_\psi\taking\vSup{\psi}\to\vinDem{\psi}.$$ 
The purpose of defining the set $\inDem{\psi}$ of internal demands above is that the supplier assignment sends $\outp{Z}$ into it, i.e. we have $s_\psi\big|_{\outp{Z}}\taking\outp{Z}\to\inSup{\psi}$ by the non-instantaneity requirement. It induces $\pi_{s_\psi\big|_{\outp{Z}}}\taking\vset_{\inSup{\psi}}\to\vset_{\outp{Z}}$. Applying $\List$ gives a 0-historical function $\List(\pi_{s_\psi\big|_{\outp{Z}}})$ which we abbreviate as
$$S''\taking\vinSup{\psi}\to\voutp{Z}.$$
Thus $S'$ and $S''$ first shuttle from supply lines to all demand lines, and then focus on only a subset of them. 
Let $\delta_\psi^1\in\Hist^1(\vset_{\Del{\psi}},\vset_{\Del{\psi}})$ be the 1-moment delay. Note that if $\Del{\psi}=\emptyset$ then $\delta_\psi^1\taking\{*\}\to\{*\}$ carries no information and can safely be ignored.

We now define the remaining helper functions:
\begin{align}\label{dia:helper functions}
E_{\psi,g}&:=(g\times\delta_\psi^1),\\\nonumber
C_{\psi,g}(\ell)&:=
\begin{cases}
[\;]&\tn{ if }|\ell|=0\\
(\ell, E_{\psi,g}\circ S'_\psi\circ C_{\psi,g}(\partial\ell))&\tn{ if }|\ell|\geq 1.
\end{cases}
\end{align}
The last is an inductive definition, which we can rewrite for $|\ell|\geq 1$ as 
$$C_{\psi,g}=\big(\id_{\vinp{Z}}\times(E_{\psi,g}\circ S'_\psi\circ C_{\psi,g}\circ\partial)\big)\circ\Delta,$$
where $\Delta:\vinp{Z}\to\vinp{Z}\times\vinp{Z}$ is the diagonal map.  Intuitively it says that a list of length $t$ on the input wires will produces a list of length $t$ on all supply wires. By Lemma \ref{lemma:little facts on historicality} $E_{\psi,g}$ is 1-historical and $C_{\psi,g}$ is $0$-historical.

We are ready to define the 1-historical function
\begin{align}\label{dia:formula for P on morphisms}
\mcP(\psi)(g)=S''_\psi\circ E_{\psi,g}\circ S'_\psi\circ C_{\psi,g}.
\end{align}

\end{announcement}

\begin{remark}
 The definitions of $S'_\psi$ and $E_{\psi,g}$ above implicitly make use of the ``zipwith" functions 
 \[\zipwith\taking\vinp{Z}\times_\NN\vinDem{\psi}\Too{\iso}\vDem{\psi}\quad\text{and}\quad\zipwith\taking\vinp{Y}\times_\NN\vDel{\psi}\Too{\iso}\vinDem{\psi},\]
 respectively.  In section \ref{sec:algebra requirements} we will make similar abuses in the calculations; however, when commutative diagrams are given, the zipwith is made ``explicit" by writing an equality between products of streams and streams of products when we mean that$\zipwith$should be applied to a product of streams. 
\end{remark}

\subsection{Running example to ground ideas and notation regarding $\mcP$}\label{sec:ground P}

In this section we compose elementary morphisms and apply them to a simple ``addition" propagator to construct a propagator that outputs the Fibonacci sequence. Let $X,Y, Z\in\Ob(\mcW)$ and $\phi\taking X\to Y$ and $\psi\taking Y\to Z$ be as in (\ref{dia:grounding objects}) and (\ref{dia:grounding morphisms}). Let $N=(\NN,1)\in\Set_*$ denote the set of natural numbers with basepoint 1. We recall the shapes of $X, Y$, and $Z$ here, but draw them with different labels:
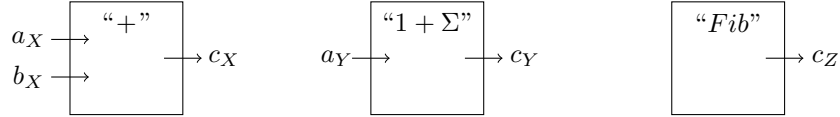
\begin{center}
\begin{tikzpicture}
\path (0,0);
  \blackbox{(1.5,1.5)}{2}{1}{$\pls$}{.5}
  \draw (-.55,1) node{$a_X$};
  \draw (-.55,.5) node{$b_X$};
  \draw (2.05,.75) node{$c_X$};
\path (4,0);
  \blackbox{(1.5,1.5)}{1}{1}{$``1+\Sigma"$}{.5}
  \draw (6.05,.75) node{$c_Y$};
  \draw (3.55,.75) node{$a_Y$};
\path (8,0);
  \blackbox{(1.5,1.5)}{0}{1}{$``Fib"$}{.5}  
    \draw (10.05,.75) node{$c_Z$};
\end{tikzpicture}
\end{center}

We have replaced the symbol $X$ with the symbol $\pls$ because we are about to define an $X$-shaped propagator $\pls\in\mcP(X)$. Given an incoming list of numbers on wire $a_X$ and another incoming list of numbers on wire $b_X$, it will create a list of their sums and output that on $c_X$. More precisely, we take $\pls\taking\List(N\times N)\to\List(N)$ to be the 1-historical propagator defined as follows. Suppose given a list $\ell\in\List(N\times N)$ of length $t$, say
$$\ell=\big[\ell_{a}(1),\ell_{a}(2),\ldots,\ell_{a}(t),\big]\zipwith\big[\ell_{b}(1),\ell_{b}(2),\ldots,\ell_{b}(t)\big]$$
Define $\pls(\ell)\in\List(N)$ to be the list whose $n$th entry (for $1\leq n\leq t+1$) is
$$\pls(\ell)(n)=
\begin{cases}
1&\tn{ if }n=1\\
\ell_{a}(n-1)+\ell_{b}(n-1)&\tn{ if }2\leq n\leq t+1
\end{cases}
$$
So for example $``+"([4,5,6,7]\zipwith[1,1,3,7])=[1,5,6,9,14]$.

We will use only this $\pls$ propagator to build our Fibonacci sequence generator. To do so, we will use wiring diagrams $\phi$ and $\psi$, whose shapes we recall here from (\ref{dia:grounding morphisms}) above.
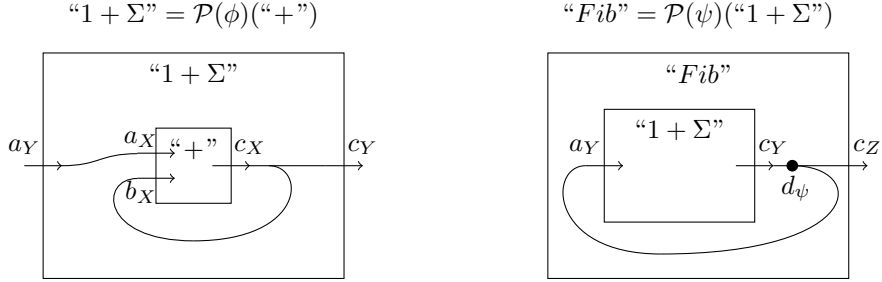
\begin{center}
\begin{tikzpicture}
\path (0,0);
  \blackbox{(4,3)}{1}{1}{$``1+\Sigma"$}{0.5}
  \draw (-.25,1.75) node{$a_Y$};
  \draw (4.25,1.75) node{$c_Y$};
\path (1.5,1);
  \blackbox{(1,1)}{2}{1}{$\pls$}{0.5}  
  \draw (1.3,1.85) node{$a_X$};
  \draw (1.3,1.15) node{$b_X$};
  \draw (2.75,1.75) node{$c_X$};
 \directarc{(0.25,1.5)}{(1.25,1.66)}
 \directarc{(2.75,1.5)}{(3.75,1.5)}
 \draw (3.00,1.5) .. controls (3.5,1.5) and (3.5,.5) .. (2,.5);
 \draw (1.25,1.33) .. controls (.75,1.33) and (.75,.5) .. (2,.5);
 \draw (2,3.5) node{$``1+\Sigma"=\mcP(\phi)(\pls)$};
\end{tikzpicture}
\hspace{.75in}
\begin{tikzpicture}
\path (0,0);
  \blackbox{(4,3)}{0}{1}{$``Fib"$}{0.5}
    \draw (4.25,1.75) node{$c_Z$};
\path (0.75,0.75);
  \blackbox{(2,1.5)}{1}{1}{$``1+\Sigma"$}{0.5}  
  \draw (.5,1.75) node{$a_Y$};
  \draw (3.0,1.75) node{$c_Y$};
  \draw (3.3,1.2) node{$d_\psi$};
 \directarc{(3,1.5)}{(3.75,1.5)}
 \filldraw[black] (3.25,1.5) circle (2pt);
 \draw (3.25,1.5) .. controls (4.25,1.5) and (4.25,.325) .. (1.5,.325);
 \draw (.5,1.5) .. controls (0,1.5) and (0,.325) .. (1.5,.325);
 \draw (2,3.5) node{$``Fib"=\mcP(\psi)(``1+\Sigma")$};
\end{tikzpicture}
\end{center}
The $Y$-shaped propagator $``1+\Sigma"=\mcP(\phi)(\pls)\in\mcP(Y)$ will have the following behavior: given an incoming list of numbers on wire $a_Y$, it will return a list of their running totals, plus 1. More precisely $``1+\Sigma"\taking\List(N)\to\List(N)$ is the 1-historical propagator defined as follows. Suppose given a list $\ell\in\List(N)$ of length $t$, say $\ell=[\ell_1,\ell_2,\ldots,\ell_t]$. Then $``1+\Sigma"(\ell)$ will be the list whose $n$th entry (for $1\leq n\leq t+1$) is 
\begin{align}\label{dia:1+sigma}
``1+\Sigma"(\ell)(n)=1+\sum_{i=1}^{n-1}\ell_i.
\end{align}
But this is not by fiat---it is calculated using the formula given in Announcement \ref{ann:P on morphisms}. We begin with the following table.\label{delme}
\begin{align*}\small
 \begin{array}{| l || l | l | l |}
 \bhline
 \multicolumn{4}{|c|}{\tn{Calculating }``1+\Sigma"}\\\bhline
 &&&\\%
 {\ell\in\ol{a_Y}}&C_{\phi,\pls}(\ell)\in\ol{\{a_Y,c_X\}}&S'_\phi C_{\phi,\pls}(\ell)\in\ol{\{a_X,b_X\}}&E_{\phi,\pls}S'_\phi C_{\phi,\pls}(\ell)\in\ol{c_X}\\\bbhline
 [\;]&[\;]&[\;]&[1]\\\hline
 [\ell_1]&[\ell_1]\tbzipwith[1]&[\ell_1]\tbzipwith[1]&[1,1+\ell_1]\\\hline
 [\ell_1,\ell_2]&[\ell_1,\ell_2]\tbzipwith[1,1+\ell_1]&[\ell_1,\ell_2]\tbzipwith[1,1+\ell_1]&[1,1+\ell_1,1+\ell_1+\ell_2]\\\hline
 [\ell_1,\ell_2,\ell_3]&[\ldots,\ell_3]\tbzipwith[\ldots,1+\ell_1+\ell_2]&[\ldots,\ell_3]\tbzipwith[\ldots,1+\ell_1+\ell_2]&[\ldots,1+\ell_1+\ell_2+\ell_3]\\\hline
 [\ell_1,\ldots,\ell_t]&[\ldots,\ell_t]\tbzipwith[\ldots,1+\sum_{i=1}^{t-1}\ell_i]&[\ldots,\ell_t]\tbzipwith[\ldots,1+\sum_{i=1}^{t-1}\ell_i]&[\ldots,1+\sum_{i=1}^t\ell_i]\\\bhline
 \end{array}
 \end{align*}
 where the last row can be established by induction. The ellipses ($\ldots$) in the later boxes indicate that the beginning part of the sequence is repeated from the row above, which is a consequence of the fact that the formulas in Announcement \ref{ann:P on morphisms} are historical. We need only calculate 
\begin{align*}
``1+\Sigma"(\ell)=\mcP(\phi)(\pls)(\ell)&=S''_\phi\circ E_{\phi,\pls}\circ S'_\phi\circ C_{\phi,\pls}(\ell)\\
&=\left[1,1+\ell_1,1+\ell_1+\ell_2,\ldots,1+\sum_{i=1}^t\ell_i\right],
\end{align*}
just as in (\ref{dia:1+sigma}).

The $Z$-shaped propagator $``Fib"=\mcP(\psi)(``1+\Sigma")\in\mcP(Z)$ will have the following behavior: with no inputs, it will output the Fibonacci sequence 
$$``Fib"()=[1,1,2,3,5,8,13\ldots].$$
Again, this is calculated using the formula given in Announcement \ref{ann:P on morphisms}. We note first that since $\inp{Z}=\emptyset$ we have $\vset_{\inp{Z}}=\{*\}$, so $\ol{\inp{Z}}=\List(\vset_{\inp{Z}})=\List(\{*\})$.

As above we provide a table that shows the calculation given the formula in Announcement \ref{ann:P on morphisms}. 
\begin{align*}
 \begin{array}{| l || l | l | l |}
 \bhline
 \multicolumn{4}{|c|}{\tn{Calculating }``Fib"}\\\bhline
 {}&C_{\psi,``1+\Sigma"}(\ell)&S'_\psi C_{\psi,``1+\Sigma"}(\ell)&E_{\psi,``1+\Sigma"}S'_\psi C_{\psi,``1+\Sigma"}(\ell)\\
 \ell\in\ol{\emptyset}&\hspace{.4in}\in\ol{\{c_Y,d_\psi\}}&\hspace{.55in}\in\ol{\{a_Y,d_\psi\}}&\hspace{1.1in}\in\ol{\{c_Y,d_\psi\}}\\\bbhline
 [\;]&[\;]&[\;]&[1]\tbzipwith[1]\\\hline
 [*]&[1]\tbzipwith[1]&[1]\tbzipwith[1]&[1,2]\tbzipwith[1,1]\\\hline
 [*,*]&[1,2]\tbzipwith[1,1]&[1,1]\tbzipwith[1,2]&[1,2,3]\tbzipwith[1,1,2]\\\hline
 [*,*,*]&[1,2,3]\tbzipwith[1,1,2]&[1,1,2]\tbzipwith[1,2,3]&[1,2,3,5]\tbzipwith[1,1,2,3]\\\hline
 [*,*,*,*]&[1,2,3,5]\tbzipwith[1,1,2,3]&[1,1,2,3]\tbzipwith[1,2,3,5]&[1,2,3,5,8]\tbzipwith[1,1,2,3,5]\\\bhline
 \end{array}
 \end{align*}
In the case of a list $\ell\in\List(\{*\})$ of length $t$, we have 
$$``Fib"(n)=\mcP(\psi)(``1+\Sigma")(\ell)=\left[1,1,2,3,\ldots,1+\sum_{i=1}^{t-2}``Fib"(i)\right].$$
Thus we have achieved our goal. Note that, while unknown to the authors, the fact that $``Fib"(t)=1+\sum_{i=1}^{t-2}``Fib"(i)$ was known at least as far back as 1891, \cite{Luc}. For us it appeared not by any investigation, but merely by cordoning off part of our original wiring diagram for $``Fib"$,
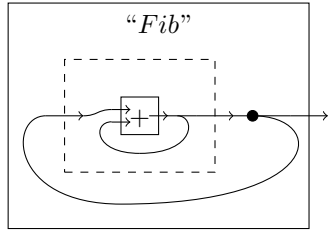
\begin{center}
 \begin{tikzpicture}
  \path (0,0);
  \blackbox{(4,3)}{0}{1}{$``Fib"$}{0.5}
  \path (0.75,0.75);
  \dashblackbox{(2,1.5)}{1}{1}{}{0.5}  
  \directarc{(3,1.5)}{(3.75,1.5)}
  \filldraw[black] (3.25,1.5) circle (2pt);
  \draw (3.25,1.5) .. controls (4.25,1.5) and (4.25,.325) .. (1.5,.325);
  \draw (.5,1.5) .. controls (0,1.5) and (0,.325) .. (1.5,.325);
  \path (1.5,1.25);
  \blackbox{(0.5,0.5)}{2}{1}{+}{0.25}  
  \directarc{(1,1.5)}{(1.375,1.5833)}
  \directarc{(2.125,1.5)}{(2.5,1.5)}
  \draw (2.25,1.5) .. controls (2.5,1.5) and (2.5,1) .. (1.75,1);
  \draw (1.375,1.42) .. controls (1.125,1.42) and (1.125,1) .. (1.75,1);
\end{tikzpicture}
\end{center}

Above in (\ref{dia:compute composition example}) we computed the supplier assignment for the composition WD, $\omega:=\psi\circ\phi\taking X\to Z$. In case the above tables were unclear, we make one more attempt at explaining how propagators work by showing a sequence of images with values traversing the wires of $\omega$ applied to $\pls$. The wires all start with the basepoint on their supply sides, at which point it is shuttled to the demand sides. It is then processed, again giving values on the supply sides that are again shuttled to the demand sides. This is repeated once more.

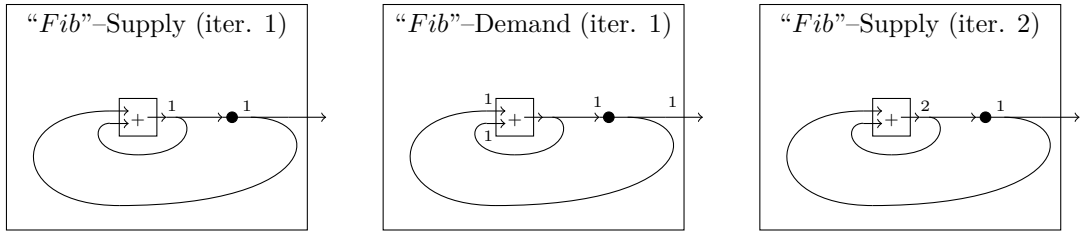
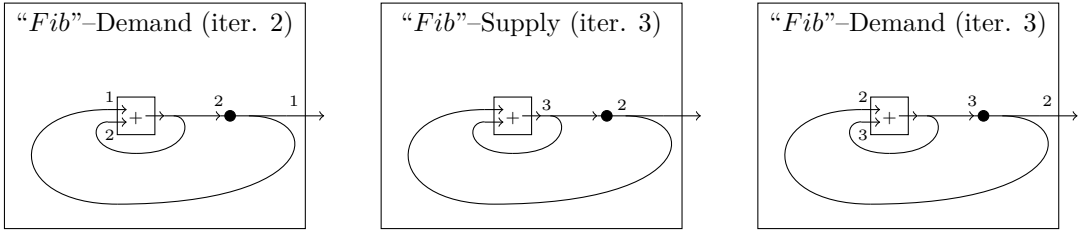
\begin{center}
\begin{tikzpicture}
  \path (0,0);
  \blackbox{(4,3)}{0}{1}{$``Fib"$--Supply (iter. 1)}{0.5}
  \directarc{(2.875,1.5)}{(3.75,1.5)}
  \draw [->] (2.125,1.5) -- (2.875,1.5);
  \filldraw[black] (3,1.5) circle (2pt);
  \draw (3.25,1.5) .. controls (4.25,1.5) and (4.25,.325) .. (1.5,.325);
  \draw (1.375,1.5833) .. controls (0,1.5833) and (0,.325) .. (1.5,.325);
  \path (1.5,1.25);
  \blackbox{(0.5,0.5)}{2}{1}{\tiny$+$}{0.25}  
  \draw (2.25,1.5) .. controls (2.5,1.5) and (2.5,1) .. (1.75,1);
  \draw (1.375,1.42) .. controls (1.125,1.42) and (1.125,1) .. (1.75,1);
  \draw (2.2,1.65) node{\tiny$1$};
  \draw (3.2,1.65) node{\tiny$1$};
\end{tikzpicture}
\hspace{.2in}
\begin{tikzpicture}
  \path (0,0);
  \blackbox{(4,3)}{0}{1}{$``Fib"$--Demand (iter. 1)}{0.5}
  \directarc{(2.875,1.5)}{(3.75,1.5)}
  \draw [->] (2.125,1.5) -- (2.875,1.5);
  \filldraw[black] (3,1.5) circle (2pt);
  \draw (3.25,1.5) .. controls (4.25,1.5) and (4.25,.325) .. (1.5,.325);
  \draw (1.375,1.5833) .. controls (0,1.5833) and (0,.325) .. (1.5,.325);
  \path (1.5,1.25);
  \blackbox{(0.5,0.5)}{2}{1}{\tiny$+$}{0.25}  
  \draw (2.25,1.5) .. controls (2.5,1.5) and (2.5,1) .. (1.75,1);
  \draw (1.375,1.42) .. controls (1.125,1.42) and (1.125,1) .. (1.75,1);
  \draw (1.4,1.75) node{\tiny$1$};
  \draw (1.4,1.25) node{\tiny$1$};
  \draw (2.85,1.7) node{\tiny$1$};
  \draw (3.85,1.7) node{\tiny$1$};
\end{tikzpicture}
\hspace{.2in}
\begin{tikzpicture}
  \path (0,0);
  \blackbox{(4,3)}{0}{1}{$``Fib"$--Supply (iter. 2)}{0.5}
  \directarc{(2.875,1.5)}{(3.75,1.5)}
  \draw [->] (2.125,1.5) -- (2.875,1.5);
  \filldraw[black] (3,1.5) circle (2pt);
  \draw (3.25,1.5) .. controls (4.25,1.5) and (4.25,.325) .. (1.5,.325);
  \draw (1.375,1.5833) .. controls (0,1.5833) and (0,.325) .. (1.5,.325);
  \path (1.5,1.25);
  \blackbox{(0.5,0.5)}{2}{1}{\tiny$+$}{0.25}  
  \draw (2.25,1.5) .. controls (2.5,1.5) and (2.5,1) .. (1.75,1);
  \draw (1.375,1.42) .. controls (1.125,1.42) and (1.125,1) .. (1.75,1);
  \draw (2.2,1.65) node{\tiny$2$};
  \draw (3.2,1.65) node{\tiny$1$};
\end{tikzpicture}
\\
\vspace{.2in}
\begin{tikzpicture}
  \path (0,0);
  \blackbox{(4,3)}{0}{1}{$``Fib"$--Demand (iter. 2)}{0.5}
  \directarc{(2.875,1.5)}{(3.75,1.5)}
  \draw [->] (2.125,1.5) -- (2.875,1.5);
  \filldraw[black] (3,1.5) circle (2pt);
  \draw (3.25,1.5) .. controls (4.25,1.5) and (4.25,.325) .. (1.5,.325);
  \draw (1.375,1.5833) .. controls (0,1.5833) and (0,.325) .. (1.5,.325);
  \path (1.5,1.25);
  \blackbox{(0.5,0.5)}{2}{1}{\tiny$+$}{0.25}  
  \draw (2.25,1.5) .. controls (2.5,1.5) and (2.5,1) .. (1.75,1);
  \draw (1.375,1.42) .. controls (1.125,1.42) and (1.125,1) .. (1.75,1);
  \draw (1.4,1.75) node{\tiny$1$};
  \draw (1.4,1.25) node{\tiny$2$};
  \draw (2.85,1.7) node{\tiny$2$};
  \draw (3.85,1.7) node{\tiny$1$};
\end{tikzpicture}
\hspace{.2in}
\begin{tikzpicture}
  \path (0,0);
  \blackbox{(4,3)}{0}{1}{$``Fib"$--Supply (iter. 3)}{0.5}
  \directarc{(2.875,1.5)}{(3.75,1.5)}
  \draw [->] (2.125,1.5) -- (2.875,1.5);
  \filldraw[black] (3,1.5) circle (2pt);
  \draw (3.25,1.5) .. controls (4.25,1.5) and (4.25,.325) .. (1.5,.325);
  \draw (1.375,1.5833) .. controls (0,1.5833) and (0,.325) .. (1.5,.325);
  \path (1.5,1.25);
  \blackbox{(0.5,0.5)}{2}{1}{\tiny$+$}{0.25}  
  \draw (2.25,1.5) .. controls (2.5,1.5) and (2.5,1) .. (1.75,1);
  \draw (1.375,1.42) .. controls (1.125,1.42) and (1.125,1) .. (1.75,1);
  \draw (2.2,1.65) node{\tiny$3$};
  \draw (3.2,1.65) node{\tiny$2$};
\end{tikzpicture}
\hspace{.2in}
\begin{tikzpicture}
  \path (0,0);
  \blackbox{(4,3)}{0}{1}{$``Fib"$--Demand (iter. 3)}{0.5}
  \directarc{(2.875,1.5)}{(3.75,1.5)}
  \draw [->] (2.125,1.5) -- (2.875,1.5);
  \filldraw[black] (3,1.5) circle (2pt);
  \draw (3.25,1.5) .. controls (4.25,1.5) and (4.25,.325) .. (1.5,.325);
  \draw (1.375,1.5833) .. controls (0,1.5833) and (0,.325) .. (1.5,.325);
  \path (1.5,1.25);
  \blackbox{(0.5,0.5)}{2}{1}{\tiny$+$}{0.25}  
  \draw (2.25,1.5) .. controls (2.5,1.5) and (2.5,1) .. (1.75,1);
  \draw (1.375,1.42) .. controls (1.125,1.42) and (1.125,1) .. (1.75,1);
  \draw (1.4,1.75) node{\tiny$2$};
  \draw (1.4,1.25) node{\tiny$3$};
  \draw (2.85,1.7) node{\tiny$3$};
  \draw (3.85,1.7) node{\tiny$2$};
\end{tikzpicture}
\end{center}
One sees the first three elements of the Fibonacci sequence $[1,1,2]$, as demanded, emerging from the output wire.

\subsection{Proof that the algebra requirements are satisfied by $\mcP$}\label{sec:algebra requirements}

Below we prove that $\mcP$, as announced, satisfies the requirements necessary for it to be a $\mcW$-algebra. Unfortunately, the proof is quite technical and not very enlightening. Given a composition $\omega=\psi\circ\phi$, there is a correspondence between the wires in $\omega$ with the wires in $\psi$ and $\phi$, as laid out in Announcement \ref{ann:composition in W}. The following proof essentially amounts to checking that, under this correspondence, the way Announcement \ref{ann:P on morphisms} instructs us to shuttle information along the wires of $\omega$ is in agreement with the way it instructs us to shuttle information along the wires of $\psi$ and $\phi$.

\begin{theorem}\label{th:main}
The function $\mcP\taking\Ob(\mcW)\to\Ob(\Sets)$ defined in Announcement \ref{ann:P on objects} and the function $\mcP\taking\mcW(Y;Z)\to\Hom_\Sets(\mcP(Y);\mcP(Z))$ given in Announcement \ref{ann:P on morphisms} satisfy the requirements for $\mcP$ to be a $\mcW$-algebra.
\end{theorem}

\begin{proof}

We must show that both the identity law and the composition law hold. This will require several technical lemmas, which for the sake of flow we have included within the current proof.

We begin with the identity law. Let $Z=(\inp{Z},\outp{Z},\vset_Z)$ be an object. The supplier assignment for $\id_Z\taking Z\to Z$ is given by the identity function
$$s_{\id_Z}\taking\outp{Z}\amalg\inp{Z}\Too{\id}\inp{Z}\amalg\outp{Z}.$$
Let $f\in\mcP(Z)=\Hist^1(\vset_{\inp{Z}},\vset_{\outp{Z}})$ be a historical propagator. We need to show that $\mcP(\id_Z)(f)=f$. 

Recall the maps 
\begin{tabbing}
\hsp\=$S_{\id_Z}\taking\vSup{\id_Z}\to\vDem{\id_Z},$
\hsp\=$S'_{\id_Z}\taking\vSup{\id_Z}\to\vinDem{\id_Z},$
\hsp\=$S''_{\id_Z}\taking\vinSup{\id_Z}\to\voutp{Z},$\\
\>$E_{\id_Z,f}\taking\vinp{Z}\to\vinSup{\id_Z},$
\>$C_{\id_Z,f}\taking\vinp{Z}\to\vSup{\id_Z},$
\end{tabbing}
from Announcement \ref{ann:P on morphisms}, where $\inSup{\id_Z}=\outp{Z}$.  
\begin{lemma}
 Suppose given a list $\ell\in\vinp{Z}$. We have 
 $$C_{\id_Z,f}(\ell)=
 \begin{cases}
 [\;]&\tn{ if }|\ell|=0,\\
 \big(\ell,f(\partial\ell)\big)&\tn{ if }|\ell|\geq 1.
 \end{cases}
 $$
\end{lemma}
\begin{proof}
 We work by induction. The result holds trivially for the empty list. Thus we may assume that the result holds for $\partial\ell$ (i.e. that $C_{\id_Z,f}(\partial\ell)=(\partial\ell,f(\partial\partial\ell)$ holds) and deduce that it holds for $\ell$. Note that $S'_{\id_Z}([\;])=[\;]$ and $E_{\id_Z,f}([\;])=f([\;])$. By the formulas (\ref{dia:helper functions}) we have  
 \begin{align*}
  C_{\id_Z,f}(\ell)
  &=\big(\id_{\vinp{Z}}\times(E_{\id_Z,f}\circ S'_{\id_Z}\circ C_{\id_Z,f}\circ\partial)\big)\circ\Delta(\ell)\\
  &=\big(\id_{\vinp{Z}}\times(E_{\id_Z,f}\circ S'_{\id_Z}\circ C_{\id_Z,f}\circ\partial)\big)(\ell,\ell)\\
  &=\big(\id_{\vinp{Z}}(\ell),E_{\id_Z,f}\circ S'_{\id_Z}\circ C_{\id_Z,f}\circ\partial(\ell)\big)\\
  &=\big(\ell,E_{\id_Z,f}\circ S'_{\id_Z}(\partial\ell,f(\partial\partial\ell))\big)\\
  &=\big(\ell,E_{\id_Z,f}\circ \pi_{\vinDem{\id_Z}}\circ S_{\id_Z}(\partial\ell,f(\partial\partial\ell))\big)\\
  &=\big(\ell,E_{\id_Z,f}\circ \pi_{\vinDem{\id_Z}}(\partial\ell,f(\partial\partial\ell))\big)\\
  &=\big(\ell,E_{\id_Z,f}(\partial\ell)\big)\\
  &=\big(\ell,f(\partial\ell)\big).\\
 \end{align*}
\end{proof}
Expanding the definition of $\mcP(\id_Z)(f)(\ell)$ we now complete the proof that the identity law holds for $\mcP$:
\begin{align*}
 \mcP(\id_Z)(f)(\ell)
 &=S''_{\id_Z}\circ E_{\id_Z,f}\circ S'_{\id_Z}\circ C_{\id_Z,f}(\ell)\\
 &=S''_{\id_Z}\circ E_{\id_Z,f}\circ S'_{\id_Z}(\ell,f(\partial\ell))\\
 &=S''_{\id_Z}\circ E_{\id_Z,f}\circ \pi_{\vinDem{\id_Z}}\circ S_{\id_Z}(\ell,f(\partial\ell))\\
 &=S''_{\id_Z}\circ E_{\id_Z,f}\circ \pi_{\vinDem{\id_Z}}(\ell,f(\partial\ell))\\
 &=S''_{\id_Z}\circ E_{\id_Z,f}(\ell)\\
 &=S''_{\id_Z}\big(f(\ell)\big)\\
 &=f(\ell).\\
\end{align*}

We now move on to the composition law. Let $s\taking m\to n$ be a morphism in $\Fin$. Let $Z\in\Ob(\mcW)$ be a black box, let $Y\taking n\to\Ob(\mcW)$ be an $n$-indexed set of black boxes, and let $x\taking m\to\Ob(\mcW)$ be an $m$-indexed set of black boxes. We must show that the following diagram of sets commutes:
$$
\xymatrix{
\mcW_n(Y;Z)\times\prod_{i\in n}\mcW_{m_i}(X_i;Y(i))\ar[d]_{\mcP}\ar[r]^{\circ_\mcW}&\mcW_m(X;Z)\ar[d]^\mcP\\
\Sets_n(\mcP(Y);\mcP(Z))\times\prod_{i\in n}\Sets_{m_i}(\mcP(X_i);\mcP(Y(i)))\ar[r]_-{\circ_\Sets}&\Sets_m(\mcP(X);\mcP(Z))
}
$$
Suppose given $\psi\taking Y\to Z$ and $\phi_i\taking X_i\to Y(i)$ for each $i$, and let $\phi=\bigotimes_i\phi_i\taking X\to Y$. We can trace through the diagram to obtain $\mcP(\psi)\circ_\Sets\mcP(\phi)$ and $\mcP(\psi\circ_\mcW\phi)$, both in $\Sets_m(\mcP(X);\mcP(Z)))$ and we want to show they are equal as functions. From here on, we drop the subscripts on $\circ_-$, i.e. we want to show $\mcP(\psi)\circ\mcP(\phi)=\mcP(\psi\circ\phi).$

Let $\omega=\psi\circ\phi$.  An element $f\in\mcP(X)=\Hist^1(\vset_{\inp{X}},\vset_{\outp{X}})$ is a 1-historical propagator, $f\taking\vinp{X}\to\voutp{X}$. We are required to show that the following equation holds in $\mcP(Z)$:
\begin{equation}\label{eq:main}
 \mcP(\psi)\circ\mcP(\phi)(f)\qeq\mcP(\omega)(f).
\end{equation}

Expanding using the definition \eqref{dia:formula for P on morphisms} of $\mcP(\psi)\circ\mcP(\phi)(f)$ and $\mcP(\omega)(f)$ we see that this translates into proving the commutativity of the following diagram:
$$\xymatrix{
\vinSup{\psi}\ar@{}[ddrrrr]|?\ar[rr]^{S''_\psi}&&\voutp{Z}&&\vinSup{\omega}\ar[ll]_{S''_\omega}
\\%
\vinDem{\psi}\ar[u]_{E_{\psi,g}}&&&&\vinDem{\omega}\ar[u]^{E_{\omega,f}}
\\%
\vSup{\psi}\ar[u]_{S'_\psi}&&\vinp{Z}\ar[ll]_{C_{\psi,g}}\ar[rr]^{C_{\omega,f}}&&\vSup{\omega}\ar[u]^{S'_\omega}
}
$$
where we abbreviated $g=\mcP(\phi)(f)$. To do so, we must prove some technical results (Lemmas \ref{lem:production lemma 1}, \ref{lem:production lemma 2}, and \ref{lem:induction step}) which assert the equality of various demand and supply streams flowing on the composed wiring diagram $\omega=\psi\circ\phi$. 

The ultimate proof of \eqref{eq:main} will be inductive in nature. That is, to prove that the result holds for a nonempty list $\ell$ of length $t\ge1$, we will assume that it holds for the list $\partial\ell$ of length $t-1$. More precisely, to prove \eqref{eq:main} we will need to know the following equality of functions $\vinp{Z}\to\vinDem{\omega}$
\begin{align}
S'_\omega\circ C_{\omega,f}=(S'_\phi\times\id)\circ (C_{\phi,f}\times\id)\circ S'_\psi\circ C_{\psi,g}
\end{align}
and this is proven by induction on the length of $\ell\in\vinp{Z}$.  The base of the induction is clear after recalling that definition \eqref{dia:helper functions} gives $C_{\omega,f}([\;])=[\;]$, $C_{\phi,f}([\;])=[\;]$ and $C_{\psi,g}([\;])=[\;]$, and that $S'_\psi$ and $s'_\omega$ are 0-historical.

The next three lemmas carry out the induction step and assume the following induction hypothesis regarding the equality of functions $\vinp{Z}\to\vinDem{\omega}$
\begin{align}\label{dia:induction hypothesis}
S'_\omega\circ C_{\omega,f}\circ\partial=(S'_\phi\times\id)\circ (C_{\phi,f}\times\id)\circ S'_\psi\circ C_{\psi,g}\circ\partial.
\end{align}

\begin{lemma}\label{lem:production lemma 1}
If we assume that equation (\ref{dia:induction hypothesis}) holds then the following diagram commutes:
$$\xymatrix{
\vinp{Z}\ar[rr]^{C_{\omega,f}}\ar[dd]_{C_{\psi,\mcP(\phi)(f)}}&&\vSup{\omega}\ar@{=}[d]
\\%
&&\vinp{Z}\times\vinSup{\phi}\times\vDel{\psi}\ar[d]^{\id\times S''_\phi\times\id}
\\%
\vSup{\psi}&&\vinp{Z}\times\voutp{Y}\times\vDel{\psi}\ar@{=}[ll]
}
$$
in other words, we have the following equality between functions $\vinp{Z}\to\vSup{\psi}$:
\begin{equation}\label{eq:cascade equality}
 C_{\psi,\mcP(\phi)(f)}=(\id\times S''_\phi\times\id)\circ C_{\omega,f}.
\end{equation}
\end{lemma}

\begin{proof}

For convenience we will abbreviate $g=\mcP(\phi)(f)$.  It follows from our induction hypothesis (\ref{dia:induction hypothesis}), the internal square in the following diagram (when composed with $(\id\times\partial)\circ\Delta\taking\vinp{Z}\to\vinp{Z}\times\vinp{Z}$) commutes:
$$
\scriptsize\xymatrix@=13pt{
\vinp{Z}\ar[dr]^{(\id\times\partial)\circ\Delta}\ar[rrrr]^{C_{\omega,f}}\ar[dddd]_{C_{\psi,g}}&&&&\vSup{\omega}\ar@{=}[d]
\\
&\vinp{Z}\times\vinp{Z}\ar[r]^{\id\times C_{\omega,f}}\ar[d]_{\id\times C_{\psi,g}}&\vinp{Z}\times\vSup{\omega}\ar[r]^{\id\times S'_\omega}&\vinp{Z}\times\vinDem{\omega}\ar@{=}[d]\ar[r]^{\id\times E_{\omega,f}}&\vinp{Z}\times\vinSup{\omega}\ar@{=}[d]
\\
&\vinp{Z}\times\vSup{\psi}\ar[d]_{\id\times S'_\psi}&&\vinp{Z}\times\vinDem{\phi}\times\vDel{\psi}\ar[r]^{\id\times E_{\phi,f}\times\delta^1_\psi}&\vinp{Z}\times\vinSup{\phi}\times\vDel{\psi}\ar[dd]_{\id\times S''_\phi\times\id}
\\
&\vinp{Z}\times\vinDem{\psi}\ar@{=}[r]\ar[d]_{\id\times E_{\psi,g}}&\vinp{Z}\times\vinp{Y}\times\vDel{\psi}\ar[r]^{\id\times C_{\phi,f}\times\id}&\vinp{Z}\times\vSup{\phi}\times\vDel{\psi}\ar[u]^{\id\times S'_\phi\times \id}&
\\
\vSup{\psi}\ar@{=}[r]&\vinp{Z}\times\vinSup{\psi}\ar@{=}[rrr]&&&\vinp{Z}\times\voutp{Y}\times\vDel{\psi}
}
$$

 The top square and left square commute by definition of $C_{\omega,f}$ and $C_{\psi,f}$ respectively, see (\ref{dia:helper functions}). The square $E_{\omega,f}=E_{\phi,f}\times\delta^1_\psi$ commutes also by definition (\ref{dia:helper functions}). The commutativity of the bottom-right corner of the diagram translates into the following identity between functions $\vinDem{\psi}\to\voutp{Y}\times\vDel{\psi}$:
\[E_{\psi,\mcP(\phi)(f)}=(S''_\phi\times\id)\circ (E_{\phi,f}\times\delta^1_\psi)\circ (S'_\phi\times\id)\circ (C_{\phi,f}\times\id).\]
 But this is a direct consequence of the definitions $E_{\psi,\mcP(\phi)(f)}=\mcP(\phi)(f)\times\delta^1_\psi$ and $\mcP(\phi)(f)=S''_\phi\circ E_{\phi,f}\circ S'_\phi\circ C_{\phi,f}$.  It follows that the outer square commutes.

 \end{proof}
 
  \begin{lemma}\label{lem:production lemma 2}
   If we assume that equation (\ref{dia:induction hypothesis}) holds, then so does the following equality of functions $\vinp{Z}\to\vinSup{\phi}$:\;
   \ffootnote{-2pt}
   {It is possible for one to draw a diagram representing this equation as we did in the preceding lemma, however we did not find such a diagram enlightening in this case.}
   \begin{align}\label{dia:what we want}
   \pi_{\vinSup{\phi}}\circ C_{\omega,f}=\pi_{\vinSup{\phi}}\circ C_{\phi,f}\circ\pi_{\vinp{Y}}\circ S'_\psi\circ C_{\psi,g}.
   \end{align}
  \end{lemma}
  \begin{proof}
\comment{gross picture
$$
\scriptsize\xymatrix@=13pt{
\vinp{Z}\ar[dr]^{(\id\times\partial)\circ\Delta}\ar[rrrr]^{C_{\omega,f}}\ar[dddd]_{C_{\psi,g}}&&&&\vSup{\omega}\ar@{=}[d]
\\
&\vinp{Z}\times\vinp{Z}\ar[r]^{\id\times C_{\omega,f}}\ar[d]_{\id\times C_{\psi,g}}&\vinp{Z}\times\vSup{\omega}\ar[r]^{\id\times S'_\omega}&\vinp{Z}\times\vinDem{\omega}\ar@{=}[d]\ar[r]^{\id\times E_{\omega,f}}&\vinp{Z}\times\vinSup{\omega}\ar@{=}[d]
\\
&\vinp{Z}\times\vSup{\psi}\ar[d]_{\id\times S'_\psi}&&\vinp{Z}\times\vinDem{\phi}\times\vDel{\psi}\ar[r]^{\id\times E_{\phi,f}\times\delta^1_\psi}&\vinp{Z}\times\vinSup{\phi}\times\vDel{\psi}\ar[dd]_{\pi}
\\
&\vinp{Z}\times\vinDem{\psi}\ar@{=}[r]\ar[d]_{\id\times E_{\psi,g}}&\vinp{Z}\times\vinp{Y}\times\vDel{\psi}\ar[r]^{\id\times C_{\phi,f}\times\id}&\vinp{Z}\times\vSup{\phi}\times\vDel{\psi}\ar[u]^{\id\times S'_\phi\times \id}\ar[dl]_{\pi}&
\\
\vSup{\psi}\ar@{=}[r]\ar[d]^{S'_\psi}&\vinp{Z}\times\vinSup{\psi}&\vSup{\phi}\ar[r]^{S'_\phi}&\vinDem{\phi}\ar[r]^{E_{\phi,f}}&\vinSup{\phi}
\\
\vinDem{\psi}\ar[d]_{\pi}&\vinp{Y}\times\vinp{Y}\ar[r]^{\id\times C_{\phi,f}}&\vinp{Y}\times\vSup{\phi}\ar[u]_{\pi}\ar[r]^{\id\times S'_\phi}&\vinp{Y}\times\vinDem{\phi}\ar[r]^{\id\times E_{\phi,f}}&\vinp{Y}\times\vinSup{\phi}\ar@{=}[d]\ar[u]_{\pi}
\\
\vinp{Y}\ar[ur]_{(\id\times\partial)\circ\Delta}\ar[rrrr]_{C_{\phi,f}}&&&&\vSup{\phi}
}
$$}
  
   We will use the following three ``forgetful" equations,
   \begin{align}
    \label{dia:forgetful1}E_{\phi,f}\circ\pi_{\vinDem{\phi}}\circ S'_\omega&=\pi_{\vinSup{\phi}}\circ(E_{\phi,f}\times\delta_\psi^1)\circ S'_\omega,\\\label{dia:forgetful2}
    \pi_{\vinSup{\phi}}\circ E_{\omega,f}\circ S'_\omega\circ C_{\omega,f}\circ\partial&=\pi_{\vinSup{\phi}}\circ\big(\id_{\vinp{Z}}\times(E_{\omega,f}\circ S'_\omega\circ C_{\omega,f}\circ\partial)\big)\circ\Delta,\\\label{dia:forgetful3}
    E_{\phi,f}\circ S'_\phi\circ C_{\phi,f}\circ\partial&=\pi_{\vinSup{\phi}}\circ\big(\id_{\vinp{Y}}\times(E_{\phi,f}\circ S'_\phi\circ C_{\phi,f}\circ\partial)\big)\circ\Delta,\\\label{dia:forgetful4}
    S'_\phi\circ C_{\phi,f}\circ\pi_{\inp{Y}}&=\pi_{\vinDem{\phi}}\circ (S'_\phi\times\id)\circ(C_{\phi,f}\times\id).  
   \end{align}
    which are ``obvious" in the sense that they are simply a matter of tracking coordinate projections.
   The proof will go as follows. We apply $E_{\phi,f}\circ\pi_{\vinDem{\psi}}$ to both sides of the assumed equality (\ref{dia:induction hypothesis}) and simplify. On the left-hand side we use (\ref{dia:forgetful1}) then the fact that by definition we have 
   \begin{align}\label{dia:E fact}
   E_{\omega,f}=E_{\phi,f}\times\delta^1_\psi,
   \end{align}
  then (\ref{dia:forgetful2}), then the definition of $C_{\omega,f}$ which we reproduce here:
   \begin{align}\label{dia:P fact}
   C_{\omega,f}=\big(\id_{\vinp{Z}}\times(E_{\omega,f}\circ S'_\omega\circ C_{\omega,f}\circ\partial)\big)\circ\Delta
   \end{align}
   to obtain the following equality of functions $\vinp{Z}\to\vinSup{\phi}$:
   \begin{align*}
    &E_{\phi,f}\circ\pi_{\vinDem{\phi}}\circ S'_\omega\circ C_{\omega,f}\circ\partial\\
    &\quad=^{(\ref{dia:forgetful1})}\pi_{\vinSup{\phi}}\circ(E_{\phi,f}\times\delta_\psi^1)\circ S'_\omega\circ C_{\omega,f}\circ\partial\\
    &\quad=^{(\ref{dia:E fact})}\pi_{\vinSup{\phi}}\circ E_{\omega,f}\circ S'_\omega\circ C_{\omega,f}\circ\partial\\
    &\quad=^{(\ref{dia:forgetful2})}\pi_{\vinSup{\phi}}\circ \big(\id_{\vinp{Z}}\times(E_{\omega,f}\circ S'_\omega\circ C_{\omega,f}\circ\partial)\big)\circ\Delta\\
    &\quad=^{(\ref{dia:P fact})}\pi_{\vinSup{\phi}}\circ C_{\omega,f}.\\
   \end{align*}
   On the right hand side we use (\ref{dia:forgetful4}), then commute the $\partial$, then apply (\ref{dia:forgetful3}),
   and then the definition of $C_{\phi,f}$ which we reproduce here:
   \begin{align}\label{dia:P fact 2}
   C_{\phi,f}=\big(\id_{\vinp{Y}}\times(E_{\phi,f}\circ S'_\phi\circ C_{\phi,f}\circ\partial)\big)\circ\Delta
   \end{align}
   to obtain the following equality of functions $\vinp{Z}\to\vinSup{\phi}$:
   \begin{align*}
    \quad&E_{\phi,f}\circ\pi_{\vinDem{\phi}}\circ (S'_\phi\times\id)\circ (C_{\phi,f}\times\id)\circ S'_\psi\circ C_{\psi,g}\circ\partial\\
    &\quad=^{(\ref{dia:forgetful4})}E_{\phi,f}\circ S'_\phi\circ C_{\phi,f}\circ\pi_{\vinp{Y}}\circ S'_\psi\circ C_{\psi,g}\circ\partial\\
    &\quad=E_{\phi,f}\circ S'_\phi\circ C_{\phi,f}\circ\partial\circ\pi_{\vinp{Y}}\circ S'_\psi\circ C_{\psi,g}\\\
    &\quad=^{(\ref{dia:forgetful3})}\pi_{\vinSup{\phi}}\circ \big(\id_{\vinp{Y}}\times(E_{\phi,f}\circ S'_\phi\circ C_{\phi,f}\circ\partial)\big)\circ\Delta\circ\pi_{\vinp{Y}}\circ S'_\psi\circ C_{\psi,g}\\
    &\quad=^{(\ref{dia:P fact 2})}\pi_{\vinSup{\phi}}\circ C_{\phi,f}\circ\pi_{\vinp{Y}}\circ S'_\psi\circ C_{\psi,g}.\\
   \end{align*}
   Combining these computations with the induction hypothesis (\ref{dia:induction hypothesis}) gives the result:
   \begin{align*}
    \pi_{\vinSup{\phi}}\circ C_{\omega,f}
    &=E_{\phi,f}\circ\pi_{\vinDem{\phi}}\circ S'_\omega\circ C_{\omega,f}\circ\partial\\
    &=E_{\phi,f}\circ\pi_{\vinDem{\phi}}\circ (S'_\phi\times\id)\circ (C_{\phi,f}\times\id)\circ S'_\psi\circ C_{\psi,g}\circ\partial\\
    &=\pi_{\vinSup{\phi}}\circ C_{\phi,f}\circ\pi_{\vinp{Y}}\circ S'_\psi\circ C_{\psi,g}.
   \end{align*}
  \end{proof}
  
  \begin{lemma}[Main Induction Step]\label{lem:induction step}
   If we assume that equation (\ref{dia:induction hypothesis}), reproduced here
   \begin{align}\tag{\ref{dia:induction hypothesis}}
    S'_\omega\circ C_{\omega,f}\circ\partial=(S'_\phi\times\id)\circ (C_{\phi,f}\times\id)\circ S'_\psi\circ C_{\psi,g}\circ\partial,
   \end{align}
   holds, then equation (\ref{dia:induction hypothesis}) holds without the precomposed $\partial$, i.e. we have the following equality of functions $\vinp{Z}\to \vinDem{\omega}$:
   \begin{align*}
    S'_\omega\circ C_{\omega,f}=(S'_\phi\times\id)\circ (C_{\phi,f}\times\id)\circ S'_\psi\circ C_{\psi,g}.
   \end{align*}
  \end{lemma}
  \begin{proof}
   To keep the notation from becoming too cluttered we adopt the following convention: an identity map written as the right hand term of a product will always mean $\id_{\vDel{\psi}}$, while an identity map written as the left hand term of a product will mean one of $\id_{\vinp{Y}}$, $\id_{\voutp{Y}}$, $\id_{\vinp{Z}}$, or $\id_{\voutp{Z}}$, which one should be clear from the context.

 The proof will be by cases, we show for each $j\in\inDem{\omega}$ that 
  \[\pi_j\circ (S'_\phi\times\id\big)\circ (C_{\phi,f}\times\id)\circ S'_\psi\circ C_{\psi,g}=\pi_j\circ S'_\omega\circ C_{\omega,f},\]
  i.e. we show that the two ways of producing internal demand streams agree by checking wire by wire. Since $\inDem{\omega}=\inDem{\phi}\amalg\Del{\psi}$, there are three main cases to consider: $j\in\Del{\psi}$, $j\in\inDem{\phi}$ with $s_\phi(j)\in\inp{Y}$, and $j\in\inDem{\phi}$ with $s_\phi(j)\in\inSup{\phi}$. We go through these in turn below.  Most of the necessary equalities will use that shuttling streams between outputs and inputs does not change the value stream.
 \begin{enumerate}
  \item Suppose $j\in\Del{\psi}$.  We use Lemma~\ref{lem:production lemma 1} and the fact that the right hand identity maps are $\id_{\vDel{\psi}}$ to see 
  \begin{align*}
   &\pi_j\circ (S'_\phi\times\id)\circ (C_{\phi,f}\times\id)\circ S'_\psi\circ C_{\psi,g}\\
   &\quad=^{(\ref{eq:cascade equality})}\pi_j\circ (S'_\phi\times\id)\circ (C_{\phi,f}\times\id)\circ S'_\psi\circ(\id\times S''_\phi\times\id)\circ C_{\omega,f}\\
   &\quad=\pi_j\circ S'_\psi\circ(\id\times S''_\phi\times\id)\circ C_{\omega,f}\\
   \tag{$*$}&\quad=\pi_{s_\psi(j)}\circ(\id\times S''_\phi\times\id)\circ C_{\omega,f}.
  \end{align*}
  Now there are two cases depending on what has supplied wire $j$.
  \begin{itemize}
   \item Suppose $s_\psi(j)\in\inp{Z}\amalg\Del{\psi}$.  Notice that in this case \eqref{dia:composition in W} gives $s_\psi(j)=s_\omega(j)$.  Then ($*$) above becomes
   \begin{align*}
    \pi_{s_\psi(j)}\circ(\id_{\vinp{Z}}\times S''_\phi\times\id_{\vDel{\psi}})\circ C_{\omega,f}
    &=\pi_{s_\psi(j)}\circ C_{\omega,f}=\pi_{s_\omega(j)}\circ C_{\omega,f}\\
    &=\pi_j\circ S'_\omega\circ C_{\omega,f}.\\
   \end{align*}
   \item Suppose $s_\psi(j)\in\outp{Y}$.  In this case \eqref{dia:composition in W} gives $s_\phi\circ s_\psi(j)=s_\omega(j)$.  Because $S''_\phi=\pi_{s_\phi\big|_{\outp{Y}}}\taking\vinSup{\phi}\to\voutp{Y}$, we see that  ($*$) simplifies as
   \begin{align*}
    \pi_{s_\psi(j)}\circ(\id\times S''_\phi\times\id)\circ C_{\omega,f}
    &=\pi_{s_\psi(j)}\circ S''_\phi\circ\pi_{\vinSup{\phi}}\circ C_{\omega,f}\\
    &=\pi_{s_\phi\circ s_\psi(j)}\circ C_{\omega,f}=\pi_{s_\omega(j)}\circ C_{\omega,f}\\
    &=\pi_j\circ S'_\omega\circ C_{\omega,f}.\\
   \end{align*}
  \end{itemize}
  
  \item Suppose $j\in\inDem{\phi}$ and $s_\phi(j)\in\inp{Y}$.  We will use Lemma~\ref{lem:production lemma 1} and the equation
  \[\pi_j\circ (S'_\phi\times\id)=\pi_{s_\phi(j)}.\]
  We will also use the fact that $\pi_{s_\phi(j)}\circ (C_{\phi,f}\times\id)=\pi_{s_\phi(j)}$, which holds because $s_\phi(j)\in\inp{Y}$ and $C_{\phi,f}$ is the identity on $\vinp{Y}$. With these in hand we compute:
  \begin{align*}
   &\pi_j\circ (S'_\phi\times\id)\circ (C_{\phi,f}\times\id)\circ S'_\psi\circ C_{\psi,g}\\
   &\quad=^{(\ref{eq:cascade equality})}\pi_j\circ (S'_\phi\times\id)\circ (C_{\phi,f}\times\id)\circ S'_\psi\circ(\id\times S''_\phi\times\id)\circ C_{\omega,f}\\
   &\quad=\pi_{s_\phi(j)}\circ (C_{\phi,f}\times\id)\circ S'_\psi\circ(\id\times S''_\phi\times\id)\circ C_{\omega,f}\\
   &\quad=\pi_{s_\phi(j)}\circ S'_\psi\circ(\id\times S''_\phi\times\id)\circ C_{\omega,f},\\
   \tag{$**$}&\quad=\pi_{s_\psi\circ s_\phi(j)}\circ(\id\times S''_\phi\times\id)\circ C_{\omega,f},\\
  \end{align*}
  There are again two cases to consider depending on what has supplied wire $j$: 
 \begin{itemize}
  \item Suppose $s_\psi\circ s_\phi(j)\in\inp{Z}\amalg\Del{\psi}$.  Then we get 
  \[\pi_{s_\psi\circ s_\phi(j)}\circ(\id_{\vinp{Z}}\times S''_\phi\times\id_{\vDel{\psi}})=\pi_{s_\psi\circ s_\phi(j)}.\]
  Now \eqref{dia:composition in W} implies the identity $s_\psi\circ s_\phi(j)=s_\omega(j)$ and thus ($**$) becomes
  \begin{align*}
   \pi_{s_\psi\circ s_\phi(j)}\circ(\id\times S''_\phi\times\id)\circ C_{\omega,f}
   &=\pi_{s_\psi\circ s_\phi(j)}\circ C_{\omega,f}=\pi_{s_\omega(j)}\circ C_{\omega,f}\\
   &=\pi_j\circ S'_\omega\circ C_{\omega,f}.\\
  \end{align*}
  \item Suppose $s_\psi\circ s_\phi(j)\in\outp{Y}$.  Then notice that by \eqref{dia:composition in W} we have $s_\omega(j)=s_\phi\circ s_\psi\circ s_\phi(j)$ and ($**$) simplifies as
  \begin{align*}
   \pi_{s_\psi\circ s_\phi(j)}\circ(\id\times S''_\phi\times\id)\circ C_{\omega,f}
   &=\pi_{s_\psi\circ s_\phi(j)}\circ S''_\phi\circ\pi_{\vinSup{\phi}}\circ C_{\omega,f}\\
   &=\pi_{s_\phi\circ s_\psi\circ s_\phi(j)}\circ C_{\omega,f}=\pi_{s_\omega(j)}\circ C_{\omega,f}\\
   &=\pi_j\circ S'_\omega\circ C_{\omega,f}.\\
  \end{align*}
 \end{itemize}
 
  \item Suppose $j\in\inDem{\phi}$ and $s_\phi(j)\in\inSup{\phi}$.  As usual we have $\pi_j\circ S'_\phi=\pi_{s_\phi(j)}$, but noting that $\vset_j=\vset_{s_\phi(j)}$, the assumptions on $j$ imply that we have
  \[\pi_j\circ S'_\phi=\pi_{s_\phi(j)}\circ\pi_{\vinSup{\phi}}.\]
  In this case \eqref{dia:composition in W} gives $s_\omega(j)=s_\phi(j)$ and thus by Lemma \ref{lem:production lemma 2}, 
 \begin{align*}
  &\pi_j\circ (S'_\phi\times\id)\circ (C_{\phi,f}\times\id)\circ S'_\psi\circ C_{\psi,g}\\
  &\quad=\pi_j\circ S'_\phi\circ C_{\phi,f}\circ \pi_{\vinp{Y}}\circ S'_\psi\circ C_{\psi,g}\\
  &\quad=\pi_{s_\phi(j)}\circ\pi_{\vinSup{\phi}}\circ C_{\phi,f}\circ \pi_{\vinp{Y}}\circ S'_\psi\circ C_{\psi,g}\\
  &\quad=^{(\ref{dia:what we want})}\pi_{s_\phi(j)}\circ\pi_{\vinSup{\phi}}\circ C_{\omega,f}\\
  &\quad=\pi_{s_\omega(j)}\circ C_{\omega,f}\\
  &\quad=\pi_j\circ S'_\omega\circ C_{\omega,f}.\\
  \end{align*}
  
 \end{enumerate}
  \end{proof}
  
  To complete the proof of Theorem~\ref{th:main} recall that we have been given morphisms $\phi\taking X\to Y$ and $\psi\taking Y\to Z$ and $\omega=\psi\circ\phi$ in $\mcW$ with notation as in Announcement \ref{ann:composition in W}.  These have corresponding supplier assignments $s_\phi, s_\psi$, and $s_\omega$. Abbreviate $g=\mcP(\phi)(f)\taking\vinp{Y}\to\voutp{Y}$. Consider the following diagram of sets:

$$\xymatrix{
\vinSup{\psi}\ar[rr]^{S''_\psi}&&\voutp{Z}&&\vinSup{\omega}\ar[ll]_{S''_\omega}
\\%
&\voutp{Y}\times\vDel{\psi}\ar@{=}[ul]&&\vinSup{\phi}\times\vDel{\psi}\ar[ll]_{S''_\phi\times\id}\ar@{=}[ur]
\\%
\vinDem{\psi}\ar[uu]_{E_{\psi,g}}\ar@{=}[dr]&&&&\vinDem{\omega}\ar[uu]^{E_{\omega,f}}
\\%
&\vinp{Y}\times\vDel{\psi}\ar[uu]^{g\times\delta^1_\psi}\ar[r]^{C_{\phi,f}\times\id}&\vSup{\phi}\times\vDel{\psi}\ar[r]^-{S'_\phi\times\id}&\vinDem{\phi}\times\vDel{\psi}\ar[uu]_{E_{\phi,f}\times\delta^1_\psi}\ar@{=}[ru]
\\%
\vSup{\psi}\ar[uu]_{S'_\psi}&&\vinp{Z}\ar[ll]_{C_{\psi,g}}\ar[rr]^{C_{\omega,f}}&&\vSup{\omega}\ar[uu]^{S'_\omega}
}
$$

Recall that our goal was to show that the outermost square commutes.  We will see that each inner square is commutative in the sense that the following equations hold:
\begin{align*}
S''_\omega=S''_\psi\circ(S''_\phi\times\id)&\taking\vinSup{\phi}\times\vDel{\psi}\too\voutp{Z}
\\%
E_{\psi,g}=g\times\delta^1_\psi&\taking \vinDem{\psi}\too\voutp{Y}\times\vDel{\psi}
\\%
g\times\delta^1_\psi=(S''_\phi\times\id)\circ(E_{\phi,f}\times\delta^1_\psi)\circ(S'_\phi\times\id)\circ(C_{\phi,f}\times\id)&\taking\vinp{Y}\times\vDel{\psi}\to\voutp{Y}\times\vDel{\psi}
\\%
E_{\phi,f}\times\delta^1_\psi=E_{\omega,f}&\taking\vinDem{\phi}\times\vDel{\psi}\too\vinSup{\omega}
\\%
S'_\omega\circ C_{\omega,f}=(S'_\phi\times\id)\circ (C_{\phi,f}\times\id)\circ S'_\psi\circ C_{\psi,g}&\taking\vinp{Z}\too\vinDem{\omega}
\\%
\end{align*}\normalsize

The first follows from Lemma \ref{lemma:pushout lemma}, especially (\ref{dia:pushout morphisms}), and Announcement \ref{ann:composition in W}, especially (\ref{dia:composition in W}). The next three follow directly from definitions (\ref{dia:helper functions}).  The last equality has been proven in Lemma~\ref{lem:induction step}.

It follows that the equation below holds for functions $\vinp{Z}\too\voutp{Z}$:
$$\mcP(\omega)(f)=S''_\omega\circ E_{\omega,f}\circ S'_\omega\circ C_{\omega,f}=S''_\psi\circ E_{\psi,g}\circ S'_\psi\circ C_{\psi,g}=\mcP(\psi)(g)=\mcP(\psi)\circ\mcP(\phi)(f)
$$
Indeed, the left-hand equality and the second-to-last equality are by definition of $\mcP$ on morphisms, as given in (\ref{dia:formula for P on morphisms}). The second equality is found by a diagram chase using the six equations above.   
\end{proof}

\section{Future work}\label{sec:future work}

The authors hope that this work can be put to use rather directly in modeling and design applications. The relationship between the operad $\mcW$ and its algebra $\mcP$ is quite explicitly a relationship between form and function. The ability to zoom in and out, i.e. to change levels of abstraction with ease is a facility which we believe is essential to any good theory of the brain, computer programs, cyber-physical systems, etc. 

Below we will discuss some possibilities for future work. We see three major directions in which to go. The first is to connect this work to other work on wiring diagrams. The second is to consider applications, e.g. to computer science and cognitive neuroscience. The third is to investigate the notion of dependency, or {\em cause and effect}, in our formalism. We discuss these in turn below.

\subsection{Connecting to other work on wiring diagrams}

While wiring diagrams have been useful in engineering for many years, there are a few mathematical approaches that should connect to our own, including \cite{AADF}, \cite{BB}, \cite{DL}, and \cite{Sp2}. 

The work by \cite{AADF} studies dynamics inside of strongly connected (transitive) networks of identical units.  Their main aim is to relate the dynamics on the network to properties of the underlying network architecture.  The underlying network should be viewed as analogous to a morphism $\psi$ in $\mcW$, while the dynamics lying over the network should be viewed as analogous to the morphism $\mcP(\psi)$.  The cells in their networks are considered to have internal states which collude with the inputs to produce the output of a cell.  There exists an algebra over $\mcW$ of ``propagators with internal states" and a retract from this algebra to $\mcP$, which should allow the transfer of results of \cite{AADF} to our framework.  Arguably one of the main aims of \cite{AADF} is to introduce a notion of inflation for these networks.  A careful comparison, see for example \cite[Figure 15]{AADF} and \cite[Figure 29]{AADF}, reveals that their inflation procedure is a special case of the composition of morphisms in $\mcW$ where the black boxes being inserted into a wiring diagram come from a special class called inflations.

In \cite{BB}, the authors investigate reaction networks and in particular stochastic Petri nets. There, various species (e.g. chemicals or populations) interact in prescribed ways, and the dynamics of their changing populations are studied. A similar but more complex situation is studied in \cite{DL}. Both of these papers work with continuous time processes, whereas we work with discrete time processes. Still, we plan to investigate the relationship between these ideas in the future.

The only other place, other than the present paper, where operads are explicitly mentioned in the context of wiring diagrams seems to be \cite{Sp2}, where the author studies systems of interacting relations using an operad $\mcT$. One might think that an operad functor would appropriately relate it to the present operad $\mcW$, but that does not appear to be the case because of the delay nodes that exist in $\mcW$ but not $\mcT$. Instead, these two operads need to be compared via a third, in which delay nodes do not occur, but wires are still directed. We hope to make this precise in the future.

\subsection{Applications, e.g. to computer science and cognitive neuroscience}

The authors' primary purposes in the above work was to formalize what we considered fundamental principles in the relation of form and function in both computers and brains. On the operad/form level we are speaking of hierarchical chunking; on the algebra/function level we are speaking of historical propagators. 

One can ask several interesting questions at this point. For example, can we create from $\mcW$ and $\mcP$ a viable computer programming language? We would hope that the propagators given by {\em computable functions} are closed in, i.e. form a subalgebra of, $\mcP$. But perhaps one could ask for more as well. For example, if each transistor in a computer acts like a NOR gate, one could ask whether or not the subalgebra generated by NOR gates is Turing complete. We conjecture that something like this is true. If so, we believe our language will provide a simple, grounded, and useful perspective on the actual operation of computers. 

There are also many interesting questions on the neuroscience side that motivated this work.  These essentially amount to a question of ``what".  What is a neuron?  What is a brain?  What is the relationship between the actions of individual neurons and the brain as a whole?  It is easy to imagine that a neuron is simply a black box where we assign certain multisets of neurotransmitters to each input and output, the historical propagators would then record activity patterns of discretized neurons.  If this turns out to be the case then the distinction between neuron and brain becomes blurred, each is simply a black box with some specified inputs and outputs.  From this perspective the questions of how the activity of individual neurons relates to the activity of a functional brain region or of the entire brain becomes subsumed by the operad formalism where we can think of each as a different choice of chunking within a single (massively complex) wiring diagram representing the connections occurring within an entire brain.  Deep questions regarding precisely how the actions of neurons in one part of the brain influence the activity in other areas will rely on the work of neuroscientists' understanding of the precise wiring pattern of the brain and remain to be understood.  We will speak more on these questions of dependency within our formalism in the next section.

\subsection{Investigating the notion of dependency}

Given a propagator with $m$-inputs and $n$-outputs, one may ask about the relation of {\em dependency} between them. When one says that the outcome of a process is dependent on the inputs, this should mean that changing the inputs will cause a change in the outputs. 

In one form or another, the ability to track changes as they propagate through a network of processes is one of the basic questions in almost any field of research. Indeed, concern with notions of {\em cause and effect} is an essential characteristic of human thought. Making mathematical sense of this notion would presumably be immensely valuable. In particular, it should have direct applications to neuroscience and computer programming disciplines. 

It is not clear that there exists a reasonable notion of causality that is {\em algebraic} in nature, i.e. one that can be formulated as a $\mcW$-algebra receiving a morphism from $\mcP$. In that case we may look to other approaches, e.g. that of Bayesian networks as in \cite{Pea} and \cite{Fon}. Whether Bayesian networks also form an algebra on $\mcW$ or a related operad, and how such an algebra compares with $\mcP$ should certainly be investigated.

\bibliographystyle{amsalpha}

\end{document}